\documentclass[preprint,review]{elsarticle}
\usepackage{comment}
\usepackage{bbm}
\usepackage{amsmath,amssymb,amsthm}
\usepackage[title]{appendix}
\usepackage{color}
\usepackage{xcolor}
\usepackage{graphicx}
\usepackage{enumerate}
\usepackage{ifpdf}
\usepackage{algorithmic}
\usepackage{float}

\usepackage{listings}
\usepackage{bbm}
\usepackage{bm}
\usepackage{mathrsfs}
\usepackage{amsmath}
\usepackage{enumerate}
\usepackage{multirow,booktabs}
\usepackage{caption}
\usepackage{makecell}
\usepackage{epstopdf}
\usepackage[percent]{overpic}
\usepackage{graphicx,graphics,epsfig}
\usepackage{mathrsfs}
\usepackage{arydshln}
\usepackage{lipsum}
\usepackage{enumerate}
\usepackage{color}
\usepackage[normalem]{ulem}
\usepackage{soul}
\usepackage{algorithm}
\usepackage[utf8]{inputenc}  % 如果你用的是 PDFLaTeX
\usepackage[T1]{fontenc}

\usepackage{multirow,booktabs}
\usepackage{makecell}
\usepackage{subfigure}

\makeatletter\@addtoreset{equation}{section} \makeatother
\newtheorem{theorem}{Theorem}[section]
\newtheorem{lemma}{Lemma}[section]

\newtheorem{definition}{Definition}[section]

\def\R{\mathbb{R}}

\def\W{\mathbb{W}}

\def\S{\mathcal{S}}

\def\T{\mathcal{T}}
\def\d{\mathrm{d}}
\def\D{\mathcal{D}}

\def\qed{~\relax\ifmmode\hskip2em \Box
 \else\unskip\nobreak\hskip1em \hfill$\Box$
 \fi \newline}

\def\R{\mathbb{R}}
\def\S{\mathbb{S}}
\def\D{\mathcal{D}}
\def\x{\bm{x}}
\def\y{\bm{y}}

\def\O{\mathcal{O}}
\def\Z{\mathbb{Z}}
\def\Zd{\Z^d}
\def\T{\mathbb{T}}
\def\Td{\T^d}
\def\TN{\mathcal{T}_N}

\def\J{\mathcal{J}}
\def\JN{\mathcal{J}_N}

\def\ftil{\widetilde{f}}
\def\fhat{\widehat{f}}
\def\phat{\widehat{\psi}}

\def\chat{\widehat{\chi_N}}
\def\d{\text{d}}
\def\i{\text{i}}
\def\k{\bm{k}}
\def\j{\bm{j}}
\def\h{\bm{h}}
\def\c{\bm{c}}
\def\s{\bm{s}}
\def\bl{\bm{\ell}}
\def\bnu{\bm{\nu}}
\def\bpsi{\bm{\psi}}

\def\bAlp{\bm{\alpha}}

\def\eikx{e^{\i\k\cdot\x}}

\def\mix{\text{mix}}
\def\sumkzd{\sum_{\k\in\Zd}}

\newmuskip\pFqmuskip

\newcommand*\pFq[6][8]{%
	\begingroup % only local assignments
	\pFqmuskip=#1mu\relax
	\mathchardef\normalcomma=\mathcode`,
	\mathcode`\,=\string"8000
	\begingroup\lccode`\~=`\,
	\lowercase{\endgroup\let~}\pFqcomma
	{}_{#2}F_{#3}{\left(\genfrac..{0pt}{}{#4}{#5};#6\right)}%
	\endgroup
}
\newcommand{\pFqcomma}{{\normalcomma}\mskip\pFqmuskip}

\usepackage[percent]{overpic}
\usepackage{tikz}
\usepackage{pgfplots} % <- Preamble
\pgfplotsset{compat=1.18}
\usetikzlibrary{decorations.pathreplacing}

\graphicspath{{Figures_PeriKer/}}

\usepackage{geometry}
\begin{document}
\begin{frontmatter}
\title{Quasi-interpolation using generalized Gaussian kernels}
%\tnoteref{label1}}
%\tnotetext[label1]{The first author was supported in part by NSFC (No. 12271002). The third author was supported in part by NSFC (No. 12101310), and the Fundamental Research Funds for the Central Universities (No. 30923010912).}

\author{Wenwu Gao\fnref{label2}}
\ead{wenwugao528@163.com}

\author{Le Hu\fnref{label2}}
\ead{hule1907356451@163.com}

\author{Zhengjie Sun\fnref{label3}\corref{cor1}}
\ead{zhengjiesun@njust.edu.cn}
\cortext[cor1]{Corresponding author}

\author{Changwei Wang\fnref{label2}}
\ead{dieck-w-mathematics@outlook.com}

 \address[label2]{School of Big Data and Statistics, Anhui University, Hefei, China}
 \address[label3]{School of Mathematics and Statistics, Nanjing University of Science and Technology, Nanjing, China}

\begin{abstract}
This paper focuses on developing a framework for constructing  quasi-interpolation with the highest achievable approximation order from generalized Gaussian kernels with the help of kernel restriction trick and periodization technique. We first demonstrate that when we restrict generalized Gaussian kernels satisfying generalized Strang-Fix conditions of order $s$ over a torus, the corresponding restricted kernels in tensor-product forms fulfill periodic Strang-Fix conditions of the same order $s$. Then, based on these restricted kernels,  we construct a periodic quasi-interpolant in Schoenberg's form and derive its error estimates for periodic function approximation  over a torus, which reveals that our quasi-interpolant attains the highest approximation order $s$. Finally, using the  periodization technique, we extend the periodic quasi-interpolant to its nonperiodic counterpart with the highest approximation order $s$ for approximating a general function defined  over a cube via a torus-to-cube transformation. This result stands in stark contrast to classical quasi-interpolation counterparts, which often yield much lower approximation orders than those dictated by the generalized Strang-Fix conditions of  generalized Gaussian kernels. Furthermore, we propose a sparse grid counterpart for high-dimensional function approximation to alleviate the curse of dimensionality. Numerical simulations confirm that our quasi-interpolation scheme is simple and computationally efficient.
\end{abstract}


\end{frontmatter}
\section{Introduction}
Quasi-interpolation is a well-established tool in function approximation and its related areas \cite{buhmann2022quasi,buhmann2024new, 
dung2020dimension,kammerer2021high,  kolomoitsev2023sparse,kolomoitsev2021quasi,  nasdala2021efficient, ortmann2024high,peigney2021fourier,
potts2016sparse,sun2024energy}.
It provides an approximation via a weighted average of sampling data without the need to solve any large-scale linear system of equations \cite{ramming_2018mcom_kernel}. In addition, it possesses optimality and regularization properties \cite{gao2020optimality}. More importantly, one can even construct a quasi-interpolant such that it preserves   inner structures of the target function, for example, Wu and Schaback \cite{wu1994shape}  constructed a quasi-interpolant  preserving positivity,  monotonicity, convexity, while reference \cite{gao2022divergence} constructed divergence-free quasi-interpolation for vector-valued function approximation. 

Given a target function $f$, a set of functions $\Phi_{\bm{j}}$ and sampling data $\lambda_{\bm{j}}(f), \bm{j}\in \bm{J}$, quasi-interpolation takes a general form 
$Qf=\sum_{\bm{j}\in \bm{J}}\lambda_{\bm{j}}(f)\Phi_{\bm{j}}$. In addition,  to derive corresponding approximation orders, $Qf$ is usually required to reproduce polynomial of certain order, which can  be either imposed conditions on  $\Phi_{\bm{j}}$ or $\lambda_{\bm{j}}(f)$ or even both \cite{buhmann2022quasi}.   The best investigated quasi-interpolation maybe the Schoenberg's model taking the form \cite{rabut1992introduction}
$ Q_hf=\sum_{\bm{j}\in\Z^d}f(\bm{j}h)\Psi(\cdot/h-\bm{j})$   
with point evaluation functionals $\lambda_{\bm{j}}(f)=f(\bm{j}h)$ and a set of functions generated by dilation and translation of a single kernel $\Psi$, i.e., $\Phi_{\bm{j}}=\Psi(\cdot/h-\bm{j})$. Moreover, Jia and Lei \cite{jia1993new} even showed that approximation order of $Q_h$ is completely characterized by Strang-Fix conditions of $\Psi$ \cite{strang1971fourier}. More precisely, $Q_h$ provides an approximation error bounded by powers of $\mathcal{O}(h^{s})$ if and only if $\Psi$ satisfies Strang-Fix conditions of order $s$ for some positive integer $s$. There have been many pioneering works on constructing kernel $\Psi$ with high-order Strang-Fix conditions, see, \cite{buhmann2015pointwise,buhmann1995quasi,ortmann2024high,ortmann2025quasi,schaback1997construction} for example. 
However, Strang-Fix conditions rule out many popular radial functions for quasi-interpolation \cite{fasshauer2007meshfree,zongmin1997compactly}. Nevertheless, radial functions have been a useful tool for multivariate (high-dimensional) function approximation including interpolation \cite{schaback_1999mcom199_improved,wendland2004scattered, ye_MAMS2019_generalized, ye_ACHA2023_positive}, as well as quasi-interpolation  \cite{buhmann1995quasi,franz2023multilevel,jeong2021approximation, ortmann2024high, sun_mcom2021_probabilistic,usta2018multilevel}. 

 Buhmann, Dyn and Levin \cite{buhmann1995quasi} initiated the idea of taking a linear combination of  translations of radial functions to form the quasi-Lagrange functions satisfying high-order Strang-Fix conditions such that the resulting quasi-interpolant  in the Schoenberg's form provides the highest attainable approximation order.  Buhmann and Dai \cite{buhmann2015pointwise} derived pointwise error estimates of quasi-interpolation with radial functions in terms of local polynomial reproduction.  Recently,  Ortmann and Buhmann \cite{ortmann2024high, ortmann2025quasi} constructed a quasi-interpolant  using   generalized multiquadrics and generalized thin plate splines.  More importantly, they showed that the resulting quasi-interpolation achieves the optimal or near-optimal approximation order within the corresponding principal shift-invariant subspace. But the method are only applicable for the radial function whose generalized Fourier transform has a singularity at origin.   Fasshauer and Zhang \cite{fasshauer2007iterated} constructed  iterated quasi-interpolation with radial kernel that converges to corresponding radial basis function interpolation with respect to the number of iterations. The iterated quasi-interpolation reformulated in matrix form can be interpreted  as Laurent expansion of the inversion of the interpolation matrix. Therefore,  the involved radial kernel is required to be positive definite with interpolation matrix whose spectrum is larger than one and less than two to ensure convergence of the asymptotic expansion.  

Under the framework of approximate approximation \cite{maz2007approximate}, Maz'ya and his coauthors studied quasi-interpolation with  scaled radial kernels satisfying  (high-order) moment conditions \cite{ma2009approximation,maz1996approximate,maz2001quasi,maz1999construction}. The resulting quasi-interpolant in Schoenberg's form can provide \textbf{approximate} convergence order $s$ up to a prescribed saturation error if the radial kernel satisfies moment conditions of order $s$. They even showed that one can appropriately choose the scale parameter such that the saturation error is negligible to the precision of computer.  But there is no asymptotic results in the classical sense since $L_\infty$-approximation errors do not converge to zero. 

Light and Cheney \cite{light1992quasi} constructed a quasi-interpolant   with radial kernels satisfying  polynomial reproduction in the convolutional sense \cite{cheney2009course}.  Wu and Liu \cite{wu2004generalized}  introduced the concept of generalized Strang-Fix condition and  constructed a quasi-interpolant  with radial kernels for scattered data.   The generalized Strang-Fix condition only requires the first part of Strang-Fix conditions and thus can be easily satisfied by most kernels including radial kernels.  Moreover, Gao and Wu \cite{gao2017constructing} even proposed a general approach for constructing radial kernels with high-order generalized Strang-Fix conditions. As a byproduct, they also showed that moment condition, polynomial reproduction in the convolutional sense, and generalized Strang-Fix condition are equivalent to each other.    Gao et al. \cite{gao2020multivariate} constructed Monte-Carlo approximation based on quasi-interpolation with radial kernels satisfying high-order generalized Strang-Fix condition. 

However,  quasi-interpolation  with radial kernels satisfying generalized Strang-Fix conditions in Schoneberg's form often provides much lower approximation orders. To improve approximation accuracy, some authors constructed multilevel quasi-interpolation with radial kernels. Usta and Levesley \cite{usta2018multilevel} constructed multilevel quasi-interpolation using Gaussian kernels on sparse grids. Recently, Franz and Wendland \cite{franz2023multilevel} proposed a convergent and efficient multilevel quasi-interpolation scheme with compactly supported radial kernels and showed  that it  converges linearly in the number of levels. Later,   Wendland and his coauthors extended the multilevel approach  to high-dimensional function approximation  \cite{kempf2023high} and even manifold-valued function approximation \cite{sharon2023multiscale}. The multilevel technique can improve approximation accuracy of quasi-interpolation, but it still can not increase the approximation order to the order of  generalized Strang-Fix conditions. 

Motivated by above discussions, a question arises: can we construct a quasi-interpolation scheme with a radial kernel in the Schoenberg's form 
that achieves the \textbf{highest} approximation order  (that is up to the order of the generalized Strang-Fix conditions satisfied by the radial kernel) ? In this paper, we provide an affirmative answer by developing a quasi-interpolation method based on   Gaussian kernel due to three observations. First, Gaussian kernel is one of  the most popular and important radial kernels. Second, Gaussian kernel has  an explicit form of  Fourier transform, which facilitates us to   derive the order of generalized Strang-Fix condition satisfied by (generalized) Gaussian kernel. More importantly, we can verify that the  restriction of (generalized) Gaussian kernel to a circle  satisfies periodic Strang-Fix conditions of the same order (see Theorem \ref {thm:Integral} and Theorem \ref{PSoPS}).

We begin by demonstrating that when we restrict a two-dimensional generalized Gaussian kernel satisfying a certain order of generalized  Strang-Fix conditions   to a circle,   the  resulting kernel satisfies periodic Strang-Fix condition of the same order.  Then, we apply the tensor-product technique to extend the restricted generalized Gaussian kernel to a multivariate counterpart, from which,   we  construct a periodic quasi-interpolation  in the Schoenberg's model  given sampling data at regularly spaced  centers  over a torus. More importantly, we establish that this periodic quasi-interpolation provides the  highest achievable approximation order.  Finally,  we  extend the methodology to non-periodic function approximation  over a cube by using the periodization technique \cite{nasdala2021efficient}, as is outlined in \cite{gao2024quasi}.

 The paper is organized as follows. Section \ref{Sec:Prelim} introduces key concepts relevant to this work, including the Wiener space, periodic Strang-Fix conditions and radial kernels. Section \ref{sec:MainResult}, divided into three subsections, develops the main quasi-interpolation scheme and derives its approximation order. Numerical simulations are presented in Section \ref{sec:NumerSimul}, while conclusions and discussions are provided in Section \ref{sec:Conclusion}.

\section{Preliminaries}
\label{Sec:Prelim}

\subsection{Wiener spaces}
We will deal with the functions in Wiener spaces. Let the $d$-dimensional  torus $\T^d$ be defined as
\begin{equation*}\label{eq:Defi_d_torus}
	\T^d:=\{\x=(x_1,\ldots,x_d)\in\R^d: |x_r|\leq \pi,~r=1,\ldots,d\}.
\end{equation*}
Then the $\bm{k}$-th Fourier coefficient of a function $f\in L_1(\T^d)$ is defined via
\begin{equation*}
\fhat(\bm{k}):=\frac{1}{(2\pi)^d}\int_{\T^d}f(\x)e^{-\i\bm{k}\cdot\x}\d\x,~~\k=(k_1,\ldots,k_d)\in\Z^d.
\end{equation*}
Let $\D(\T^d)$ be the space of $2\pi$-periodic functions with infinite smoothness, and $\D'(\T^d)$ be the dual space. Any function $f\in\D'(\T^d)$ has a Fourier series expansion
\begin{equation*}
f(\x)=\sum_{\bm{k}\in\Z^d}\fhat(\bm{k})e^{\i \bm{k}\cdot \x}.
\end{equation*}
Note that  convergence of above series is related to  decay of  Fourier coefficients, which can be characterized  in Wiener spaces \cite{folland_1999book_realAnalysis,stein_2011book_fourier}. For $1\leq q\leq \infty$, $\gamma\geq 0$, we respectively denote $A_q^{\gamma} (\T^d)$ and $A_{q,\mix}^{\gamma}(\T^d)$ the periodic (isotropic) Wiener space
\begin{equation}\label{eq:AqalpSpace}
	\begin{aligned}
		A_q^{\gamma}(\T^d):=\left\{f(\x)\in L_1(\T^d):\big\|(1+|\k|_2^2)^{\gamma/2}\fhat(\k)\big\|_{\ell_q(\Z^d)}<\infty\right\},
	\end{aligned}
\end{equation}
and the periodic (mixed) Wiener space
\begin{equation}\label{eq:AqalpMixSpace}
	A_{q,\mix}^{\gamma}(\T^d):=\left\{f(\x)\in L_1(\T^d):\big\|\prod_{r=1}^d(1+|k_r|^2)^{\gamma/2}\fhat(\k)\big\|_{\ell_q(\Z^d)}<\infty\right\}.
\end{equation}
Particularly, by taking $q=1$ and $\gamma=0$, the space $A_1^0(\T^d)$ reduces to the well-known Wiener space $A(\T^d)$ containing functions with absolutely summable Fourier series. Moreover, with $q=2$, the space $A_2^{\gamma}(\T^d)$ is exactly the Sobolev space $H_2^{\gamma}(\T^d)$ \cite{sprengel2000class}. 

The above Fourier series expansion has a discrete counterpart. More precisely, let $N$ be a positive integer and define an index set $\J_N$ by
$$\J_N=\{\bm{j}=(j_1,\ldots,j_d)\in\Z^d:~-\frac{N}{2}\leq j_r<\frac{N}{2},~r=1,\ldots,d\}.$$ 
In the sequels, we will work with the equidistant grid set $\TN=\{\frac{2\pi \bm{j}}{N},~\bm{j}\in\J_N\}$ on $\T^d$. 
Given discrete function values sampled over the set $\TN$, we use $\ftil$ to denote the discrete Fourier series expansion with the $\k$-th discrete Fourier coefficient
\begin{equation}
\label{eq:discrete_fourier_coefficien}
	\ftil_{\bm{k}}:=\frac{1}{N^d}\sum_{\bm{j}\in \JN}f\left(\frac{2\pi\bm{j}}{N}\right)e^{-2\pi\i \bm{k} \cdot \bm{j}/N}.
\end{equation}
In particular, for any $f\in A(\T^d)$, we even have the aliasing formula \cite[(1)]{sprengel2000class}
\begin{equation}\label{eq:aliasing_formula}
\ftil_{\k}=\sum_{\bl\in\Z^d}\fhat(\bm{k}+\bl N).
\end{equation}

\subsection{Periodic Strang-Fix conditions}

Periodic Strang-Fix conditions \cite{brumme1994error,kolomoitsev2020approximation,
  sickel1999some,sprengel2000class} can be regarded as the periodic counterpart of (complete) Strang-Fix conditions \cite{strang1971fourier}. Many functions satisfy periodic Strang-Fix conditions, such as  $2\pi$-periodized centered B-splines, linear combination of trigonometric polynomials \cite{kolomoitsev2020approximation}, periodization of Gaussian functions \cite{hubbert2023convergence},  and so forth.  In particular, Sprengel \cite{sprengel2000class} proposed a further improvement of periodic Strang-Fix condition as follows \cite[Definition~3]{sprengel2000class}.
\begin{definition}\label{def:PSF}
   The $2\pi$-periodic kernel $\chi_N$ satisfies the periodic Strang-Fix conditions of order $s>0$ with $q\ge 1$ and $\tau \ge 0$, if for all $\k\in\JN$ the inequalities 
    \begin{align}
     	|1-N^d \chat(\k)|&\leq b_{\bm{0}}|\k|_2^s N^{-s},\\
		|N^d \chat(\k+\bnu N)|&\leq  b_{\bnu}|\k|_2^s N^{-s-\tau},~~\bnu=(\nu_{1},\cdots,\nu_{d})\in\Z^d\backslash\{\bm{0}\}\label{eq:new_PSF_2}
    \end{align}
hold for some sequence of non-negative constants $\{b_{\bnu}\}_{\bnu \in \Z^d}$ with 
$$
|| (1+|\bnu|_2^2)^{\tau /2}b_{\bnu}||_{\ell_q(\mathbb{Z}^d)} < \infty.
$$
\end{definition}

\begin{comment}\begin{definition}{\emph{(Periodic Strang-Fix conditions)}}
\label{def:PSF}
	Let  $\{b_{\bnu}\}\in\ell_2(\Z^d)$  be some non-negative sequence and $\J_N$ be the index set defined as above.  Let $\chi_N$ be a $2\pi$-periodic continuous function. We say $\chi_N$ satisfies  periodic Strang-Fix conditions of order $s$  if  the following inequalities 
	\begin{align}
		|1-N^d \chat(\k)|&\leq b_{\bm{0}}|\k|_2^s N^{-s},\label{eq:PSF-1}\\
		|N^d \chat(\k+\bnu N)|&\leq b_{\bnu}|\k|_2^s N^{-s},~~\bnu\in\Z^d\backslash\{\bm{0}\}\label{eq:PSF-2}
	\end{align}
	hold true  for all $\k\in\JN$.
\end{definition}
\end{comment}
Moreover,  Sprengel \cite{sprengel2000class}  pointed out that if $\chi_N$ additionally satisfies the cardinal interpolation property
$$\chi_N\Big(\frac{2\pi\bm{j}}{N}\Big)=\delta_{0,\bm{j}},~\bm{j}\in\JN,$$ 
where $\delta_{0,\bm{j}}$ is a Kronecker delta symbol, i.e., 
$$  
\delta_{0,\bm{j}} = 
\begin{cases}
1, & \text{if } \bm{j} = 0, \\
0, & \text{if } \bm{j} \neq 0,
\end{cases}
$$
then the error of corresponding Lagrange interpolation 
\begin{equation}\label{eq:LagInterp}
I_Nf(\cdot)=\sum_{\bm{j}\in\JN}f\Big(\frac{2\pi\bm{j}}{N}\Big)\chi_N\Big(\cdot-\frac{2\pi\bm{j}}{N}\Big)
\end{equation}
can be completely characterized by periodic Strang-Fix conditions.  
\begin{lemma}\label{lem:ConvCardInterp}
	Let  $f\in A_q^{\mu}(\T^d)$ with $q\geq 1$, $\mu\geq \tau\geq 0$, and $\mu> d(1-1/q)$. Suppose that the kernel $\chi_N$  of the above cardinal interpolation $I_Nf$ satisfies  periodic Strang-Fix conditions of order $s> d/2$. Then for $\sigma=\min\{\mu,s+\tau\}$, there exists a constant $C_{\sigma}$ such that the inequality
	$$\|f-I_N f\|_{A_q^{\tau}(\T^d)}\leq C_{\sigma} N^{-(\sigma-\tau)}\|f\|_{A_q^{\mu}(\T^d)}$$
	holds true for all $f\in A_q^{\mu}(\T^d)$.
\end{lemma}

\subsection{Radial kernels}
For any function $f\in L_2(\R^d)$, we define its Fourier transform pair as follows
\begin{equation*}\label{eq:FourTranDef}
	\widehat{f}(\bm{\omega})=\frac{1}{(2\pi)^{d/2}}\int_{\R^d}f(\x)e^{-\i \bm{\omega}\cdot \x}\d \x, ~~f(\x)=\frac{1}{(2\pi)^{d/2}}\int_{\R^d}\widehat{f}(\bm{\omega})e^{\i \bm{\omega}\cdot \x}\d \bm{\omega}.
\end{equation*}
In particular, for a radial function 
$$\Phi(\x)=\phi(\|\x\|)\in L_1(\R^d)\cap C(\R^d),\quad \text{with}~\phi:\R^+\rightarrow \R,$$ 
its Fourier transform is also radial and can be further represented by
\begin{equation}\label{eq:RBFFourTran}
	\widehat{\Phi}(\bm{\omega})=\mathcal{F}_d\phi(r)=(2\pi)^{d/2}r^{-(d-2)/2}\int_{0}^{\infty}\phi(t)t^{d/2}J_{(d-2)/2}(rt)dt,~~r=\|\bm{\omega}\|,
\end{equation}
where $J_{\nu}(z)$ denotes the Bessel function of the first kind.

The restriction of radial kernels from high-dimensional Euclidean spaces to low-dimensional manifolds has been widely employed to construct important classes of basis functions and to analyze their approximation properties. Notable examples include spherical basis functions, obtained by restricting radial functions in $\mathbb{R}^d$ to $\mathbb{S}^{d-1}$ \cite{hubbert2002radial,narcowich2007approximation}, and zonal periodic basis functions derived through restriction to the circle \cite{fuselier2015order}. 
More generally, kernels on embedded Riemannian submanifolds are frequently constructed via restriction of radial basis functions \cite{fuselier2012scattered}. 

Specifically, given a RBF kernel $\Phi(\x,\y) = \phi(\|\x - \y\|)$ defined on $\mathbb{R}^d \times \mathbb{R}^d$ and a $d_{\mathbb{M}}$ dimensional compact manifold $\mathbb{M}$ isometrically embedded in $\mathbb{R}^d$, 
the \emph{restricted kernel} $\Psi: \mathbb{M} \times \mathbb{M} \to \mathbb{R}$ is defined as:
\begin{equation*}\label{eq:RestrictionTechnique}
    \Psi(\x,\y) := \phi\left( \|\x - \y\| \right)|_{\x,\y \in \mathbb{M}}.
\end{equation*}
This restriction preserves both the positive definiteness and smoothness properties of $\Phi$. 
Of particular relevance to our work is the case $\mathbb{M} = \mathbb{S}^{d-1}$, where Narcowich et al.\cite{narcowich2007approximation} established a fundamental connection between the Fourier transform of radial kernels and the Fourier-Legendre coefficients of the restricted spherical kernels. 
We state this relationship below, as it plays a central role in our subsequent analysis.

\begin{lemma}{(\cite[Proposition 3.1]{narcowich2007approximation})}\label{lem:FourierCoeff}
	Let $\Phi$ be a conditionally positive definite  radial function of order $k$ defined on $\R^{d}$ and $\psi(\x\cdot\y):=\Phi(\x, \y)|_{\x,\y\in\S^{d-1}}$. If the Fourier transform $\widehat{\Phi}$ is measurable on $\R^{d}$, then for $\ell\geq 2k+1$,it holds that 
	$$\widehat{\psi}(\ell)=\int_{0}^{\infty} t\widehat{\Phi}(t)J_{\nu}^2(t)\d t,~~\nu:=\ell+\frac{d-2}{2}.$$
\end{lemma}
% \begin{equation}\label{eq:ResGauss-SF2m}
	%     \psi_{2m+2}(\alpha)=\frac{1}{\sqrt{2\pi}c}\sum_{j=0}^{m}\frac{(-1)^j}{c^{2j}j!}{m+\frac{1}{2} \choose m-j}
	%     \left(2\sin^2(\frac{\alpha}{2})\right)^j\ e^{-\frac{2\sin^2(\frac{\alpha}{2})}{c^2}},    
	% \end{equation}
% which can be written as
% \begin{equation}\label{eq:ResGauss-SF2m2}
	%     \psi_{2m+2}(\alpha)=\frac{1}{\sqrt{2\pi}c}\sum_{j=0}^{m}\frac{(-1)^i}{i!}{m+\frac{1}{2} \choose m-i}
	%     \left(\frac{1-\cos\alpha}{c^2}\right)^i\ e^{-\frac{1-\cos\alpha}{c^2}}.    
	% \end{equation}

\section{High-order quasi-interpolation}
\label{sec:MainResult}
In this section, we will develop a high-order quasi-interpolation method using restricted generalized Gaussian kernels. We begin by constructing a tensor-product kernel over a torus derived from a generalized Gaussian kernel restricted to the circle, and then demonstrate that it satisfies high-order periodic Strang-Fix conditions. Next, we utilize this kernel to formulate the proposed periodic quasi-interpolation method.  We also develop a sparse grid quasi-interpolation scheme for high-dimensional function approximation to alleviate the curse of dimensionality. Finally,  we extend the quasi-interpolation scheme to non-periodic functions via the periodization technique discussed in \cite{gao2024quasi,potts2021approximation}.

\subsection{Generalized Gaussian kernels restricted  over a torus  $\mathbb{T}^d$}

 This study considers a two-dimensional generalized Gaussian kernel  defined as
\begin{equation}\label{eq:Gauss-SF2m}
	\Phi_c(\x, \y)=\phi_{2m+2}(r;c)=\frac{1}{\sqrt{2\pi}c}L_{m}^{(1/2)}\big(\frac{r^2}{2c^2}\big)\ e^{-r^2/(2c^2)}, ~\x, \y\in \mathbb{R}^ {2}, ~r=\|\x-\y\| ,
\end{equation}
where $m$ is a nonnegative integer and $c$ is a positive scale parameter. The function $L_{m}^{(\beta)}(x)$ denotes the generalized Laguerre polynomial in the form
\begin{equation}\label{eq:GeneralLaguerrePoly}
	L_{m}^{(\beta)}(x)=\sum_{k=0}^m(-1)^k \binom{m+\beta}{m-k}\frac{x^k}{k!}.
\end{equation}
Based on Gao and Wu \cite{gao2017constructing}, we can verify that $\Phi_c$ satisfies the generalized Strang-Fix condition of order $2m+2$. More importantly, we shall further verify that its restriction to a circle satisfies periodic Strang-Fix condition of the same order $2m+2$. We define $$\bar{\x}=\x/\|\x\|,~\bar{\y}=\y/\|\y\|,~\bar{r}=\|\bar{\x}-\bar{\y}\|,$$ 
and substitute them into \eqref{eq:Gauss-SF2m} yielding the formulation of the periodic function in the form
\begin{equation}\label{eq:ResGauss-SF2m}
	\psi_{2m+2}(\alpha;c)=\frac{1}{\sqrt{2\pi}c}L_{m}^{(1/2)}\Big(\frac{2\sin^2(\frac{\alpha}{2})}{c^2}\Big)\ e^{-2\sin^2(\frac{\alpha}{2})/c^2}, \alpha \in \mathbb{T},
\end{equation}
where $\alpha$ represents the  circular distance between $\bar{\x}$ and $\bar{\y}$ and $\sin^2(\frac{\alpha}{2})=\frac{\bar{r}^2}{4}$. Then, based on  Lemma \ref{lem:FourierCoeff}, we have the following theorem.
\begin{theorem}\label{thm:Integral}
	Let $\psi_{2m+2}(\alpha;c)$ be the restricted generalized Gaussian kernel defined as in  \eqref{eq:ResGauss-SF2m} and $b_0$ be a positive constant. Then we have:
	\begin{align*}
		&(i)~~|\phat_{2m+2}(\ell;c)-1|\leq b_0|\ell|^{2m+2}c^{2m+2},~\text{for fixed } \ell ~\text{and sufficiently small}~ c,\\
		&(ii)~~|\phat_{2m+2}(\ell;c)| ~\text{decays exponentially for sufficiently large} ~\ell.
	\end{align*}
\end{theorem}
\begin{proof}
    See \ref{Appendix:A} for the detailed proof.
\end{proof}

From above restricted Gaussian kernel $\psi_{2m+2}(\alpha;c)$ defined over a circle, we go further with constructing an anisotropic  multivariate kernel $\bpsi_{\h, \s}$ defined over a torus in  the tensor-product form  
\begin{equation}\label{eq:IsoTPKer}
	\bpsi_{\h,\s}(\x;\c)=\prod_{r=1}^d\psi_{h_r,s_r}(x_r;c_r),
\end{equation}
where $\h=(h_1,\cdots,h_d)$, $\s=(s_1,\cdots, s_d)$, $\c=(c_1,\cdots,c_d)$, and $$\psi_{h_r,s_r}(x_r;c_r)=\psi_{s_r}(x_r;c_r)\cdot h_r,\quad h_r=\frac{1}{N_r}.$$ Here $\psi_{s_r}(x_r;c_r)$ being defined in \eqref{eq:ResGauss-SF2m} with $s_r=2 m_r+2$ for some positive integers $N_r$ and $m_r$. 
Moreover, we shall show that $\bpsi_{\h,\s}$ satisfies periodic Strang-Fix conditions. For clarity of exposition, we only prove the isotropic case with $N_r=N$, $m_r=m$, $c_r=c=\mathcal{O}{(1/N)}$, $h_r =1/N$, $s_r=s$, for all $r=1,2,\cdots, d$.  All results extend naturally to anisotropic configurations.

\begin{theorem}\label{PSoPS}
	The isotropic counterpart of the kernel $\bpsi_{\h,\s}$  defined  in Equation \eqref{eq:IsoTPKer}  satisfies  periodic Strang-Fix conditions in Definition \ref{def:PSF} with $s=2m+2$.
\end{theorem}
\begin{proof}
	We begin with showing that $\bpsi_{\h,\s}$ satisfies the first part  of Definition \eqref{def:PSF}. According to the definition $\psi_{h,s}(x_r;c)=\psi_{s}(x_r;c)/N$,
	we have
    $$
		\big|1-N^d\widehat{\bpsi}_{\h,\s}(\k;\c)\big|=\big|1-N^d \prod_{r=1}^d\phat_{s}(k_r;c)/N^d\big|=\big|1-\prod_{r=1}^d\phat_{s}(k_r;c)\big|.
	$$
    Before proceeding, we need the following property: assuming that $a_r$ ($r=1,2,\ldots,d$) are uniformly bounded,  then there exists a constant $C>0$ such that
    $$|1-\prod_{r=1}^d a_r|=\big|\sum_{r=1}^d(1-a_r)\cdot a_{r+1}\cdots a_d\big|\leq C\sum_{r=1}^d|1-a_r|.$$
In particular, by letting $a_r=\phat_{s}(k_r;c)$, we can apply  the above inequality and  Theorem \ref{thm:Integral} to obtain
	\begin{align*}
		\big|1-\prod_{r=1}^d\phat_{s}(k_r;c)\big|\leq C\sum_{r=1}^d|1-\phat_{s}(k_r;c)|\leq Cb_0N^{-s} \sum_{r=1}^d|k_r|^s
		\leq   b_{\bm{0}}|\k|_2^sN^{-s}, \quad b_{\bm{0}}=Cb_{0}.
	\end{align*}
    
We go further with showing that $\bpsi_{\h,\s}$ satisfies the second part  of Definition \eqref{def:PSF}. Observe that
$$
|N^d\widehat{\bpsi}_{\h,\s}(\k+\bnu N; \c)| =|\prod_{r= 1}^d\phat_{s}(k_r+\nu_rN;c)|\leq\prod_{r= 1}^d|\phat_{s}(k_r+\nu_rN;c)|. 
$$
Since   $\widehat{\psi}_s(\ell;c)$ decays exponentially for large $\ell$, there exists  non-negative constants $\{b_{\nu_r}\}_{r=1}^d$ satisfying
\begin{equation*}
\label{eq:B_v}
|| (1+|\bnu|_2^2)^{\tau /2}b_{\nu_r}||_{\ell_q(\mathbb{Z}^d)} < \infty,
\end{equation*}
such that the following inequalities hold: 
    $$|\phat_{s}(k_r+\nu_r N;c)|\leq b_{\nu_r}|k_r|^{s}N^{-s-\tau }, \ r=1,\cdots, d.$$
    Namely, the second part of periodic Strang-Fix conditions (see  \eqref{eq:new_PSF_2}). 
    Moreover,  the first part of  Theorem \ref{thm:Integral} together with $|k_r|\le N/2$ yields 
    \begin{equation*}
    \label{eq:kernel_bound}
    |\phat_{s}(k_r;c)| \le 1+b_0|k_r|^sN^{-s}\le 1+b_0/2^{s}. \ 
    \end{equation*}
By setting $|\phat_{s}(k_{r^*}+\nu_{r^*} N;c)|=\max\limits_{1\leq r \leq d}\{|\phat_{s}(k_r+\nu_r N;c)| \}$, we have
\begin{align*}	
	|N^d\widehat{\bpsi}_{\h,\s}(\k+\bnu N; \c)|
	 &~\leq |\phat_{s}(k_{r^*}+\nu_{r^*} N;c)|\prod_{j\neq r^*}|\phat_{s}(k_j;c)|\\
        &~\leq  b_{\nu_{r^*}}|k_{r^*}|^sN^{-s-\tau} (1+b_0/2^{s})^{d-1}\\
         &~\leq  b_{\bnu}|\k|_2^sN^{-s-\tau},  \ \ 
	\end{align*}
where $b_{\bnu}=b_{\nu_{r^*}}(1+b_0/2^{s})^{d-1}$. 
Hence, the theorem holds with
\begin{equation*}
|| (1+|\bnu|_2^2)^{\tau /2}b_{\bnu}||_{\ell_q(\mathbb{Z}^d)}=(1+b_0/2^s)^{d-1}|| (1+|\bnu|_2^2)^{\tau /2}b_{\nu_{r^*}}||_{\ell_q(\mathbb{Z}^{d})} < \infty.
\end{equation*}
This completes the proof.
\end{proof}

Theorem \ref{PSoPS} demonstrates that the restricted generalized Gaussian kernel in tensor-product form satisfying the same order periodic Strang-Fix conditions as the one of generalized Strang-Fix conditions satisfied by generalized Gaussian kernel $\phi_{2m+2}$.
Consequently, we can apply it to construct quasi-interpolation over a torus in the following subsection.

\subsection{Quasi-interpolation  for periodic function approximation over a torus} $\mathbb{T}^d$
We first consider the data sampled at uniform grid over a torus and then extend it to   directly uniform  gird   as well.  Given the discrete function values $\{f(\x_{\j})\}_{\j\in\JN}$ sampled from a $2\pi$-periodic continuous function $f$ at the equidistant grid set $\TN:=\{\x_{\j}=\frac{2\pi \j}{N},~\j\in\J_N\}$ over $\T^d$, we first  construct an ansatz
\begin{equation}\label{eq:quasiInterp}
	Q_{\h,\s}f(\x) =  \sum_{\j\in\JN} f(\x_{\j})\bpsi_{\h,\s}(\x-\x_{\j};\c).
\end{equation}
This quasi-interpolation scheme is computationally efficient, requiring only linear combinations of data values and shifted kernels. Unlike traditional quasi-interpolation methods \cite{ jeong2021approximation,wu1994shape}, it inherently avoids the boundary problem due to the periodic setting.  We go further with deriving its error estimates. 
\begin{theorem}\label{maintheorem}
Let $\bpsi_{\h,\s}$ be an  isotropic counterpart of the kernel defined in \eqref{eq:IsoTPKer} that satisfies the periodic Strang-Fix conditions of order $s>d/2$, as specified in Definition \ref{def:PSF}. Then, under the assumptions of  Lemma  \ref{lem:ConvCardInterp} together with $\sigma=\min\{\mu,s+\tau\}$, there exists a positive constant $C$ such that
$$\|f-Q_{\h,\s}f\|_{A_q^{\tau}(\Td)}\leq CN^{-(\sigma-\tau)}~\|f\|_{A_q^{\mu}(\Td)}$$
holds for any $f \in A_q^{\mu}(\mathbb{T}^d)$.
\end{theorem}

\begin{proof}
We begin by decomposing the approximation error via the triangle inequality
\begin{align}\label{eq:interpolationError1}
   \|f-Q_{\h,\s}f\|_{A_q^\tau(\T^d)}
    \leq \|f-I_Nf\|_{A_q^\tau(\T^d)}+\|Q_{\h,\s}f-I_Nf\|_{A_q^\tau(\T^d)}.
\end{align}
Here $I_Nf$ denotes the Lagrange interpolation operator from \eqref{eq:LagInterp} with $\chi_N$ satisfying periodic Strang-Fix condition of order $s$. Note that the interpolation error $\|f-I_Nf\|_{A_q^\tau(\T^d)}$ has been given in Lemma \ref{lem:ConvCardInterp} as
\begin{equation}
    \label{eq:error_first}
\|f-I_Nf\|_{A_q^\tau(\T^d)} \le C N^{-(\sigma-\tau)}\|f\|_{A_q^\mu(\T^d)}.
\end{equation}
We only need to derive the bound of the second term on the right-hand side of \eqref{eq:interpolationError1}.
Applying Fourier series expansion to $\bpsi_{\h,\s}  (\x-\x_{\j})$ and $\chi_N(\x-\x_{\j})$, we have 
\begin{equation*}
\begin{split}
Q_{\h,\s}f(\x)-I_Nf(\x)
&=\sum_{\bm{j}\in \JN}f(\x_{\j})\sumkzd(\widehat{\bpsi}_{\h,\s}(\k)-\chat(\k))e^{i\k\cdot(\x-\x_{\j})}\\
		&=\sumkzd \big(\widehat{\bpsi}_{\h,\s}(\k)-\chat(\k)\big)\eikx\Big(\sum_{\bm{j}\in \JN}f(\x_{\j}) e^{-i\k\cdot\x_{\j}}\Big).
        \end{split}
\end{equation*}
This together with  the discrete Fourier coefficients $\ftil_{\k}$ given in \eqref{eq:discrete_fourier_coefficien} leads to
\begin{equation*}
\begin{split}
    	Q_{\h,\s}f(\x)-I_Nf(\x)
%&=\sumkzd\big(N^d\ftil_{\k}\widehat{\bpsi}_{\h,\s}(\k)-N^d\ftil_{\k}\chat(\k)\big)\eikx\\
		&=\sumkzd \ftil_{\k}\big(N^d\widehat{\bpsi}_{\h,\s}(\k)-N^d\chat(\k)\big) \eikx.
        \end{split}
\end{equation*}
To distinguish between contributions from the fundamental frequency band and those from aliased modes, we write $\k=\k_0+\bnu N$, where $\k_0\in\JN$  and $\bnu\in\Zd$. This in turn yields two distinct sums: one for $\bnu=\bm{0}$ and the other  for $\bnu\neq \bm{0}$. In particular, we have
\begin{equation*}
	\begin{aligned}
Q_{\h,\s}f(\x)-I_Nf(\x)&=\sum_{\k_0\in \JN}\ftil_{\k_0}\big(N^d\widehat{\bpsi}_{\h,\s}(\k_0)-N^d\chat(\k_0)\big) e^{i\k_0 \cdot \x}\\
		&\!\!\!\!\!+\sum_{\k_0\in \JN}\sum_{\bnu\in\Z^d\backslash \{\bm{0}\}}\ftil_{\k_0+\bnu N}\big(N^d\widehat{\bpsi}_{\h,\s}(\k_0+\bnu N)-N^d\chat(\k_0+\bnu N)\big) e^{\i(\k_0+\bnu N)\cdot\x}.
	\end{aligned} 
\end{equation*} 
We first consider $1\leq q<\infty$. 
Using the definition of the norm $\|\cdot\|_{A_q^{\tau}(\mathbb{T}^d)}$ from \eqref{eq:AqalpSpace}, we can get
\begin{equation}\label{eq:normDefnition}
	\begin{aligned}
		\|Q_{\h,\s}f&-I_Nf\|_{A_q^{\tau}(\mathbb{T}^d)}^q
=\sum_{\k_0\in \JN}\Big(\big|\ftil_{\k_0}\big(N^d\widehat{\bpsi}_{\h,\s}(\k_0)-N^d\chat(\k_0)\big)(1+|\k_0|_2^2)^{\tau /2})\big|^q\\
+&\sum_{\bnu\in\Z^d\backslash \{\bm{0}\}}\Big|\ftil_{\k_0+\bnu N}\big(N^d\widehat{\bpsi}_{\h,\s}(\k_0+\bnu N)-N^d\chat(\k_0+\bnu N)\big)(1+|\k_0+\bnu N|_2^2)^{\tau /2}\Big|^q
		\Big).
	\end{aligned}
\end{equation}
Then we apply periodic Strang-Fix conditions in Definition \ref{def:PSF} and Theorem \ref{PSoPS} to get the following bounds:
\begin{align*}
	|N^d\widehat{\bpsi}_{\h,\s}(\k_0)-N^d\chat(\k_0)| \leq|1-N^d\widehat{\bpsi}_{\h,\s}(\k_0)| + |1-N^d\chat(\k_0)|
    \leq  b_{\bm{0}}|\k_0|_2^s N^{-s},
\end{align*}
and
\begin{align*}
	|N^d\widehat{\bpsi}_{\h,\s}(\k_0+\bnu N)-N^d\chat(\k_0+\bnu N)|\leq b_{\bnu}|\k_0|_2^s N^{-s-\tau},~~\bnu\in\Z^d\backslash \{\bm{0}\}
    \end{align*}
with   
\begin{align}
    \label{eq:gamma_1}
\gamma_1:=|| (1+|\bnu|_2^2)^{\tau /2}b_{\bnu}||_{\ell_q(\mathbb{Z}^d)} < \infty.\end{align}
Substituting the above inequalities into 
\eqref{eq:normDefnition} and using the aliasing formula \eqref{eq:aliasing_formula}, we can obtain
\begin{equation*}
	\begin{aligned}
		\|Q_{\h,\s}f-I_Nf\|_{A_q^\tau(\mathbb{T}^d)}^q
		\leq N^{-(\sigma-\tau)q}\Big(b_{\bm{0}}^q\mathcal{E}_1
+\mathcal{E}_2\Big),
	\end{aligned}
\end{equation*}
where we define 
\begin{align}
    &\mathcal{E}_1:=N^{\sigma q -\tau q 
 -sq}\sum_{\k_0\in \JN}|\k_0|_2^{sq}\cdot  (1+|\k_0|_2^2)^{\tau q/2}\cdot\Big|\sum_{\bl\in\Zd}\fhat(\k_0+\bl N)\Big|^q,\label{eq:E_1}\\
 &\mathcal{E}_2:=N^{\sigma q-2\tau q-sq}\sum_{\k_0\in \JN}\sum_{\bnu\in\Z^d\backslash \{\bm{0}\}} b_{\bnu}^q  |\k_0|_2^{sq}\cdot(1+|\k_0+\bnu N|^2_2)^{\tau q/2}\cdot\Big|\sum_{\bl\in\Zd} \fhat(\k_0+\bnu N+\bl N)\Big|^q\label{eq:E_2}.
\end{align}
Next, we demonstrate that the two terms $\mathcal{E}_1$ and $\mathcal{E}_2$ can be both bounded by $\|f\|_{A_q^{\mu}(\T^d)}$ (see Lemma \ref{E_1}). Consequently, we have 
\begin{equation}
\label{eq:error_second}
   \|Q_{\h,\s}f-I_Nf\|_{A_q^\tau(\mathbb{T}^d)} \le   b_{\bm{0}}\gamma_2 \gamma_3  N^{-(\sigma-\tau) }\ \|f\|_{A_q^{\mu}(\mathbb{T}^d)}+C\gamma_1 \gamma_2 \gamma_3 N^{-(\sigma-\tau) } \|f\|_{A_q^{\mu}(\mathbb{T}^d)}.
\end{equation}
Finally, combining \eqref{eq:error_first}  with  \eqref{eq:error_second},
we arrive at the final error estimate.
The case $q=\infty$ can be derived similarly. 
\end{proof}

We go further with extending the above results to  the directly uniform grid by constructing a quasi-interpolant
\begin{equation}\label{eq:quasiInterpMixed}
	\begin{aligned}
		Q_{\h,\s}^{\mathrm{mix}}f(\x)=&~Q_{h_1,s_1}Q_{h_2,s_2}\cdots Q_{h_d,s_d}f(\x) \\
		=  &~\sum_{j_1}\sum_{j_2}\cdots\sum_{j_d} f(x_{1,j_1},\ldots,x_{d,j_d})\prod_{r=1}^d\psi_{h_r,s_r}(x_r-x_{r, j_r};c_r),
	\end{aligned}
\end{equation}
with the kernel $\psi_{h_r,s_r}(x_r;c_r)=\psi_{s_r}(x_r;c_r)\cdot h_r$ and $h_r=\frac{1}{N_r}$. Similarly, we can get the following error estimate.

\begin{theorem}\label{th:mix}
Let the quasi-interpolation $Q_{\h,\s}^{\mathrm{mix}}f$ be defined by \eqref{eq:quasiInterpMixed}, where the kernel $\psi_{h_r,s_r}(x_r;c_r)$ satisfying the periodic Strang-Fix conditions of order $s_r$, $r=1,\ldots,d$, as specified in Definition \ref{def:PSF}. Then, under the assumptions of Lemma \ref{lem:ConvCardInterp}, there exists a positive constant $C$ such that the inequality
$$\|f-Q_{\h,\s}^{\mathrm{mix}}f\|_{A_{q,\mathrm{mix}}^{\tau}(\Td)}\leq C\|f\|_{A_{q,\mathrm{mix}}^{\mu}(\Td)}\sum_{r=1}^dN_r^{-(\sigma_r-\tau)}$$ holds true for any  $f\in A_{q,\mathrm{mix}}^{\mu}$ with $\sigma_r=\min\{\mu,s_r+\tau\}$.
\end{theorem}
\begin{proof}
We first rewrite the error as
\begin{equation*}
	\begin{aligned}
		&f(\x)-Q_{\h,\s}^{\mathrm{mix}}f(\x)\\
		=&(I-Q_{h_1,s_1}Q_{h_2,s_2}\cdots Q_{h_d,s_d})f(\x)\\
		=&\sum_{r=1}^d(I-Q_{h_r,s_r})Q_{h_{r+1},s_{r+1}}\cdots Q_{h_d,s_d}f(\x).
	\end{aligned}
\end{equation*}
Using the triangle inequality and  Theorem \ref{maintheorem}, we have 
\begin{equation}\label{eq:mixnorm_err1}
\begin{aligned}
	\|f-Q_{\h,\s}^{\mathrm{mix}}f\|_{A_{q,\mathrm{mix}}^{\tau}(\Td)}
    =&~ \big\|\sum_{r=1}^d(I-Q_{h_r,s_r})Q_{h_{r+1},s_{r+1}}\cdots Q_{s_d,h_d}f\big\|_{A_{q,\mathrm{mix}}^{\tau}(\Td)}\\
    \leq &~ \sum_{r=1}^d\|(I-Q_{h_r,s_r})Q_{h_{r+1},s_{r+1}}\cdots Q_{s_d,h_d}f\big\|_{A_{q,\mathrm{mix}}^{\tau}(\Td)}\\
	\leq &~\sum_{r=1}^d C_rN_r^{-(\sigma_r-\tau)}\|Q_{h_{r+1},s_{r+1}}\cdots Q_{h_d,s_d}f\|_{A_{q,\mathrm{mix}}^{\mu}(\Td)}.
    \end{aligned}
\end{equation}
Moreover,  note that
\begin{align*}
\|Q_{h_j,s_j}f\|_{A_{q,\mathrm{mix}}^{\mu}(\Td)}
\leq &~\|Q_{h_j,s_j}f-f\|_{A_{q,\mathrm{mix}}^{\mu}(\Td)}+\|f\|_{A_{q,\mix}^{\mu}(\Td)}\\
\leq &~ C_j\|f\|_{A_{q,\mix}^{\mu}(\Td)}+\|f\|_{A_{q,\mix}^{\mu}(\Td)},\\
\leq &~ (1+C_j)\|f\|_{A_{q,\mix}^{\mu}(\Td)}     ,~j=1,\ldots,d.
\end{align*}
Using the mathematical induction,  we can prove 
$$
\|Q_{h_{r+1},s_{r+1}}\cdots Q_{h_d,s_d}f\|_{A_{q,\mathrm{mix}}^{\mu}(\Td)} \leq C \|f\|_{A_{q,\mathrm{mix}}^{\mu}(\Td)}.
$$
This together with \eqref{eq:mixnorm_err1} gives the desired result.
\end{proof}

Sparse grids have been widely used for high-dimensional function approximation to alleviate the curse of dimension \cite{gao2024quasi,griebel2005sparse,jeong2021approximation}.  We end this subsection  with constructing a sparse grid counterpart of quasi-interpolation  $Q_{\h,\s}^{\mathrm{mix}}$ defined in \eqref{eq:quasiInterpMixed}.

Let $\bm{n}=(n_1,\ldots,n_d)\in\mathbb{N}^d$ be a $d$-tupe with length $|\bm{n}|=n_1+\ldots+n_d$ and $\W_{\bm{N}}$ be a point set consisting of the directionally uniform points with the mesh size 
$$\bm{h}_{\bm{n}}:=2^{-\bm{n}}=\big(2^{-n_1},\ldots,2^{-n_d}\big).$$
Applying the union operator to point sets $\W_{\bm{n}}$, we can construct  the sparse grids $\W_{l,d}$ at level $l$ in the form 
\begin{equation*}
\W_{l,d}:=\bigcup_{|\bm{n}|=l+d-1}\W_{\bm{n}}.
\end{equation*}
Besides, we can derive the number of points over the sparse grids $\W_{l,d}$ as
\begin{equation*}
|\W_{l,d}|=(-1)^{d-1}\sum_{j=0}^{d-1}(-1)^j\binom{d-1}{j}\sum_{|\bm{n}|=l+j}|\W_{\bm{n}}|,
\end{equation*}
which is much less than the one of full grids or directly uniform grids.
Suppose  we have discrete function values sampled at the above sparse grids at hand, then following \cite{gao2024quasi},  we can construct a corresponding sparse grid quasi-interpolation
\begin{equation}\label{sparse} 
Q_{l,d}f(\x) = (-1)^{d-1}\sum_{j=0}^{d-1}(-1)^j\binom{d-1}{j}\sum_{|\bm{n}|=l+j}Q_{\h,\s}^{\mathrm{mix}}f(\x)
\end{equation}
with $Q_{\h,\s}^{\mathrm{mix}}f$ being defined in  \eqref{eq:quasiInterpMixed}, that is,
\begin{equation*}
Q_{\h,\s}^{\mathrm{mix}}f(\x)=\sum_{\x_{\bm{n},\k}\in\W_{\bm{n}}}f(\x_{\bm{n},\k})\bpsi_{\h_{\bm{n}},\s}(\x-\x_{\bm{n},\k};\c).
\end{equation*}
Furthermore, based on similar analysis provided in  reference \cite{hubbert2023convergence}, we can  expect   a convergence order $\mathcal{O}(2^{-ls^*}l^{d-1})$  of our sparse grid quasi-interpolation $Q_{l,d}f$, where $s^*=\min_{1\le r\le d}\{\sigma_r-\tau\}$.

Up to now,  we have constructed a quasi-interpolation scheme based on generalized Gaussian kernels such that it provides the highest achievable approximation order for periodic function approximation over a torus.  However,  the methodology can be easily extended  to non-periodic function approximation over a cube by using the periodization technique \cite{nasdala2021efficient} that has been recently studied in  \cite{gao2024quasi}. 
\subsection{Quasi-interpolation  for non-periodic function approximation over a cube}

 The construction process consists of three steps.  We first transfer a non-periodic target function $g$  defined over a cube  to a periodic counterpart  $f$ defined over a torus via the torus-to-cube transformation \cite{nasdala2021efficient}.  Then we apply the above constructed quasi-interpolation $Q_{\h,\s}f$ defined in \eqref{eq:quasiInterp} to $f$. Finally,  we get the resulting quasi-interpolation $Q_{\h,\s}g$  for approximating $g$ from $Q_{\h,\s}f$ with the help of the density function defined in   reference \cite{nasdala2021efficient} (see Formula \eqref{densityfunction} as an example).  Similarly, we can get a sparse grid counterpart  $Q_{l,d} g$  from  $Q_{l,d}f$ defined in \eqref{sparse}. 
 
Following the periodization technique   \cite{nasdala2021efficient},  given any point $\x=(x_1,\cdots,x_d)$ over $\mathbb{T}^d$, we first transform it to the cube $[-\pi,\pi]^d$ as $\bm{\gamma}(\x)=(\gamma_1(x_1),\cdots,\gamma_d(x_d))$ by torus-to-cube transformation $\bm{\gamma}$.
Furthermore, with the help of $\bm{\gamma}$, we transform a non-periodic target function $g$ defined over a cube to a periodic counterpart $f$ defined over  $\mathbb{T}^d$ as 
\begin{equation}\label{periodization}
f(\x)=g(\gamma_1(x_1),\cdots, \gamma_d(x_d))\prod_{r=1}^d\sqrt{\omega_r(\gamma_r(x_r))\gamma^{'}_r(x_r)}, 
\end{equation}
where the weight function and derivative of $\gamma_r(x_r)$ are defined by $\omega_r(\gamma_r(x_r))$ and  $\gamma^{'}_r(x_r)$, respectively.  Then, given a set of  grid points $\{\x_{\bold {j}}\}$ over $\T^d$,  we evaluate $\{f(\x_{\bold {j}})\}$ using Formula \eqref {periodization}. Based on these sampling data, we construct a quasi-interpolation scheme   $Q_{\h,\s}f$ defined in \eqref{eq:quasiInterp} for $f$.  This in turn leads to a  quai-interpolation scheme $Q_{\h,\s}g$ for $g$ defined  over a cube as
\begin{equation}\label{np}
Q_{\h,\s}g(\y)=\sum_{\j \in\JN }f(\x_{\j})\sqrt{\frac{\boldsymbol{\varrho}(\y)}{\bm{\omega}(\bm{y})}}\bpsi_{\h,\s}(\boldsymbol{\gamma}^{-1}(\y)-\x_{\j};\c).
\end{equation}
 Here the corresponding inverse transformation 
$
\boldsymbol{\gamma}^{-1}(\y)=(\gamma_1^{-1}(y_1),\cdots,\gamma_d^{-1}(y_d)): [-\pi, \pi]^d \rightarrow \mathbb{T}^d 
$ 
is defined in the sense of $\x=\boldsymbol{\gamma}^{-1}(\y) \in \mathbb{T}^d $ if and only if $\y=\boldsymbol{\gamma}(\x) \in [-\pi, \pi]^d $. The weight function is then defined as $\bm{\omega}(\bm{y})=\prod_{r=1}^d \omega_r(\gamma_r(x_r))$.  The non-negative $L_1$-integrable density function $\boldsymbol{\varrho}$ of $\boldsymbol{\gamma}$  is defined via
\begin{equation}\label{densityfunction}
\boldsymbol{\varrho}(\y)=\prod_{r=1}^d \varrho_r(y_r), \ \varrho_r(y_r):=\frac{1}{\gamma_r^{'}(\gamma_r^{-1}(y_r))}.
\end{equation}
Moreover, based on  Theorem $3.4$ in \cite{gao2024quasi}, we have corresponding error estimates in terms of  weighted $L_{\infty}$-norm 
$$
||Q_{\h,\s}g-g||_{L_{\infty}\left( [-\pi,  \pi]^d,\sqrt{\frac{\bm{\omega}}{\boldsymbol{\varrho}}}\right)}=||Q_{\h,\s}f-f||_\infty,
$$
where we define 
$$||Q_{\h,\s}g-g||_{L_{\infty}\left( [-\pi,  \pi]^d,\sqrt{\frac{\bm{\omega}}{\boldsymbol{\varrho}}}\right)}=\inf\big\{C:\ \big|(Q_{\h,\s}g(\bm{y})-g(\bm{y}))\sqrt{\bm{\omega}(\bm{y})/\boldsymbol{\varrho(\bm{y})}}\big|\leq C \ \text{for almost all} \ \bm{y}\big\}.$$
Similarly, we derive   $Q_{l,d}g$ based on sparse grids quasi-interpolation $Q_{l,d}f$ defined over a torus and the corresponding weighted $L_{\infty}$-norm  error estimates.  For more details, we refer readers to Subsection $3.2$ in reference \cite{gao2024quasi}. 
However, in such a case, $Q_{l,d}g$ is not sparse grids quasi-interpolation defined over a cube.  How to construct sparse grids quasi-interpolation for a nonperiodic function approximation defined over  a cube  will require more involved techniques that we shall study in the future works.
\section{Numerical examples}\label{sec:NumerSimul}
This section performs some  numerical simulations of applying our   quasi-interpolation based on restricted generalized  Gaussian kernels $\psi_{2m+2}$ and its tensor-product form in approximating periodic and non-periodic functions. Based on Theorem \ref{thm:Integral}, it is easy to verify that $\psi_{2m+2}$ satisfies   periodic Strang-Fix conditions in Definition \ref{def:PSF} with $s=2m+2$.

\textbf{Example 1.}
We consider finitely smooth periodic functions  \cite[Sec. $4$]{griebel2014fast} over $\mathbb{T}^d$,
$$F_p:\T^d\rightarrow\R: \bAlp\longmapsto\bigotimes_{r=1}^d f_p(\alpha_r),$$
where
$$f_p(\alpha_r)=\lambda_p(2+\text{sgn}(\alpha_r)\cdot \sin^p(\alpha_r))$$
with $\lambda_p$ denoting a normalization constant such that $\|f_p\|_{L_2(\T)}=1$. Note that for $\epsilon>0$, we have $f_p\in  H^{\frac{1}{2}+p-\epsilon}(\T)$ and thus $F_p\in H^{\frac{1}{2}+p-\epsilon}(\T^d)$.  In the following numerical simulations,  we set
$p=6 \geq s$ to ensure the approximation is sufficiently smooth, thereby allowing us to achieve the
expected approximation order for our quasi-interpolation.

We first approximate  the one-dimensional function  $f_6(\alpha_1)$ using our periodic quasi-interpolant with restricted generalized Gaussian kernel  $\psi_{2m+2}$.  We consider  $L_{\infty}$-norm approximation errors and corresponding convergence rates.  Numerical results under $m=0,1, 2$ (corresponding to $s=2,4,6$), are presented in Table \ref{tab:1D_FuncAppr_gauss}. It illustrates that convergence rates are consistent with the orders of periodic Strang-Fix conditions. Moreover, we also approximate a two-dimensional function $F_6(\T^2)$ on regular tensor grids and present corresponding approximation errors in Figure \ref{fig.L2conv_p=6}. It again validates the theoretical analysis. 
To evaluate the approximation errors on sparse grids, we approximate the two-dimensional function using the proposed method on sparse grids for levels $l = 7, \ldots, 11$. Corresponding posteriori error estimates are  shown in Table \ref{tab:2D_FuncAppr_gauss_periodic}. We can find that the sparse grids  quasi-interpolation can provide the same convergence rates as its full grid counterpart. 

\begin{table}[ht]
\centering
%\captionsetup{labelsep=period,labelfont=bf}
\caption{$L_{\infty}$-norm approximation errors and convergence rates.}
\label{tab:1D_FuncAppr_gauss}
\begin{tabular}{ccccccc}
\toprule
 & \multicolumn{2}{c}{$m=0$} & \multicolumn{2}{c}{$m=1$} & \multicolumn{2}{c}{$m=2$} \\
\cmidrule(lr){2-3} \cmidrule(lr){4-5} \cmidrule(lr){6-7}
$N$ & error & rate & error & rate & error & rate \\
\specialrule{1.5pt}{0pt}{0pt}
$64$   & 3.625e-03 &      & 3.201e-05 &      & 2.508e-06 &      \\
$128$  & 9.042e-04 & 2.03 & 1.957e-06 & 4.08 & 3.567e-08 & 6.20 \\
$256$  & 2.259e-04 & 2.02 & 1.217e-07 & 4.05 & 5.450e-10 & 6.14 \\
$512$  & 5.647e-05 & 2.01 & 7.594e-09 & 4.04 & 8.470e-12 & 6.10 \\
$1024$ & 1.412e-05 & 2.01 & 4.742e-10 & 4.03 & 1.517e-13 & 6.03 \\
\bottomrule
\end{tabular}
\end{table}

\begin{table}[ht]
\centering
\caption{$L_{\infty}$-norm approximation errors and convergence rates on sparse grids.}  
\label{tab:2D_FuncAppr_gauss_periodic}
\begin{tabular}{ccccccccc}
\toprule
 & \multicolumn{2}{c}{$m=0$}  & \multicolumn{2}{c}{$m=1$} & \multicolumn{2}{c}{$m=2$}  \\
\cmidrule(lr){2-3} \cmidrule(lr){4-5} \cmidrule(lr){6-7}
$l$ & error & rate & error & rate & error & rate \\
\specialrule{1.5pt}{0pt}{0pt}
$7$  & 1.806e-01 &      & 4.490e-03 &      & 9.530e-03 &      \\
$8$  & 5.811e-02 & 1.44 & 2.053e-03 & 0.99 & 1.092e-03 & 2.75 \\
$9$  & 8.064e-03 & 1.98 & 1.105e-04 & 2.36 & 6.689e-07 & 6.09 \\
$10$ & 4.121e-04 & 2.59 & 3.395e-07 & 4.02 & 4.702e-09 & 6.53 \\
$11$ & 1.804e-04 & 2.41 & 1.647e-08 & 4.34 & 1.282e-10 & 6.26 \\
\bottomrule
\end{tabular}
\end{table}

\begin{figure}[htbp]
\centering
\vspace{1cm}
\begin{tikzpicture}[scale=0.8]
	\begin{loglogaxis}[
		grid= both,
		major grid style={line width=1pt,draw=gray!50},
		grid style = dotted,
		mark size = 3pt,
		%log ticks with fixed point,
		%ytick distance = 0.001,
		xmin = 10, xmax = 10^3,
		ymin = 0.5*10^(-12),
		ymax = 2*10^(-2),
		xlabel={$N$},
		legend cell align = {left},
		legend pos = south west,
		legend entries={$m=0$,$m=1$,$m=2$}]
		\addplot  [ultra thick,dashed, mark = square*,mark options={solid},red]  coordinates { (16,7.7904e-03)(32,2.7316e-03) (64,7.5273e-04) (128,1.9298e-04) (256,4.8544e-05)(512,1.2181e-05)};
		\addplot  [ultra thick, dotted,mark = *, color = green]  coordinates { (16,5.1326e-03)(32,6.4755e-04) (64,4.9780e-05) (128,3.2796e-06) (256,2.0768e-07)(512,1.3264e-08)};
		\addplot  [ultra thick,mark = triangle*, blue]  coordinates { (16,4.5380e-03)(32,2.9480e-04) (64,7.0067e-06) (128,1.2117e-07) (256,1.9414e-09)(512,3.0518e-11)};
		\addplot[dashed, color = black] coordinates {(60,6*1e-7) (240,6*1e-7/4096) };
		\addplot[dashed, color = black] coordinates {(60,3e-4) (240,3e-4/256)};
		\addplot[dashed, color = black] coordinates {(60, 5e-3) (240,5e-3/16)};
	\end{loglogaxis}
	\node[rectangle] at (5.2,1) {\scriptsize slope 6};
	\node[rectangle] at (5.2,3.3) {\scriptsize slope 4};
	\node[rectangle] at (5.2,4.8) {\scriptsize slope 2};
\end{tikzpicture}
\hspace{0.5cm}
\begin{tikzpicture}[scale=0.8]
	\begin{loglogaxis}[
		grid=both,
		major grid style={line width=1pt,draw=gray!50},
		grid style = dotted,
		mark size = 3pt,
		xmin = 10, xmax = 10^3,
		ymin = 0.5*10^(-12),
		ymax = 2*10^(-2),
		xlabel={$N$},
		legend cell align = {left},
		legend pos = south west,
		legend entries={$m=0$,$m=1$,$m=2$}]
		\addplot  [ultra thick,dashed, mark = square*,mark options={solid},red]  coordinates { (16,1.8061e-03)(32,5.2517e-04) (64,1.3781e-04) (128,3.4900e-05) (256,8.7536e-06)(512,2.1902e-06)};
		\addplot  [ultra thick, dotted,mark = *, color = green]  coordinates { (16,1.0130e-03)(32,1.1645e-04) (64,8.8595e-06) (128,5.8274e-07) (256,3.6894e-08)(512,2.3137e-09)};
		\addplot  [ultra thick, mark = triangle*, color = blue]  coordinates { (16,8.0216e-04)(32,5.5584e-05) (64,1.3236e-06) (128,2.3317e-08) (256,3.7732e-10)(512,5.9597e-12)};
		\addplot[dashed, color = black] coordinates {(60,1.2*1e-7) (240,1.2*1e-7/4096) };
		\addplot[dashed, color = black] coordinates {(60,0.5e-4) (240,0.5e-4/256)};
		\addplot[dashed, color = black] coordinates {(60,1e-3) (240,1e-3/16)};
	\end{loglogaxis}
	\node[rectangle] at (5.2,0.7) {\scriptsize slope 6};
	\node[rectangle] at (5.2,3) {\scriptsize slope 4};
	\node[rectangle] at (5.2,4.5) {\scriptsize slope 2};
\end{tikzpicture}
\vspace{-1pt}
\caption{Approximation errors of the two-dimensional function $F_6(\T^2)$.}
\label{fig.L2conv_p=6}
\end{figure}

\begin{figure}[htbp]
\centering
\begin{tikzpicture}[scale=0.8]
	\begin{loglogaxis}[
		width=.4\textwidth,
		height=.4\textheight,
		grid=both,
		major grid style={line width=1pt,draw=gray!50},
		grid style = dotted,
		mark size = 2pt,
		%log ticks with fixed point,
		%ytick distance = 0.001,
		xtick = {10^3,10^4,10^5,10^6,10^7},
		xticklabels = {$10^3$,$10^4$,$10^5$,$10^6$,$10^7$},
		ymin = 10^(-9),
		ymax = 10^1,
		ytick = {10^(-8),10^(-6),10^(-4),10^(-2),10^0},
		yticklabels = {$10^{-8}$,$10^{-6}$,$10^{-4}$,$10^{-2}$,$10^0$},
		xlabel={$N$},
		title={$\psi_2$},
		legend cell align = {left},
		legend pos = south west,
		legend entries={$d=2$,$d=3$,$d=4$,$d=5$}]
		\addplot [thick,dashed, mark = square*,mark options={solid},red] table [x=N,y=Linf_d2] {Data_PeriodKer/sgqi_conv_diffdim_Linf_m2_F2.dat};
		\addplot [thick, dotted, mark=*,mark options={solid}, green] table [x=N,y=Linf_d3] {Data_PeriodKer/sgqi_conv_diffdim_Linf_m2_F2.dat};
		\addplot [thick,mark = triangle*, blue] table [x=N, y=Linf_d4] {Data_PeriodKer/sgqi_conv_diffdim_Linf_m2_F2.dat};
		\addplot [thick,densely dashed, mark = diamond*,mark options={solid}, magenta] table [x=N,y=Linf_d5] {Data_PeriodKer/sgqi_conv_diffdim_Linf_m2_F2.dat};
	\end{loglogaxis}
	%\draw [dashed] (2,0.4)--(5,1.6) node[below=1.5pt]{slope 4};
\end{tikzpicture}
\hspace{0.2cm}
\begin{tikzpicture}[scale=0.8]
	\begin{loglogaxis}[
		width=.4\textwidth,
		height=.4\textheight,
		grid=both,
		major grid style={line width=1pt,draw=gray!50},
		grid style = dotted,
		mark size = 2pt,
		%ymin= 1e-15,ymax = 1e-5,
		%xmin = 0, xmax = 8,
		title={$\psi_4$},
		xlabel={$N$}, %ylabel={Relative $L_2$ error},
        xtick = {10^3,10^4,10^5,10^6,10^7},
		xticklabels = {$10^3$,$10^4$,$10^5$,$10^6$,$10^7$},
		ymin = 10^(-9),
		ymax = 10^1,
		ytick = {10^(-8),10^(-6),10^(-4),10^(-2),10^0},
		yticklabels = {$10^{-8}$,$10^{-6}$,$10^{-4}$,$10^{-2}$,$10^0$},
		legend cell align = {left},
		legend pos = south west,
		legend entries={$d=2$,$d=3$,$d=4$,$d=5$}]
		\addplot [thick,dashed, mark = square*,mark options={solid},red] table [x=N,y=Linf_d2] {Data_PeriodKer/sgqi_conv_diffdim_Linf_m4_F2.dat};
		\addplot [thick, dotted, mark=*,mark options={solid}, green] table [x=N,y=Linf_d3] {Data_PeriodKer/sgqi_conv_diffdim_Linf_m4_F2.dat};
		\addplot [thick,mark = triangle*, blue] table [x=N, y=Linf_d4] {Data_PeriodKer/sgqi_conv_diffdim_Linf_m4_F2.dat};
		\addplot [thick,densely dashed, mark = diamond*,mark options={solid}, magenta] table [x=N,y=Linf_d5] {Data_PeriodKer/sgqi_conv_diffdim_Linf_m4_F2.dat};
	\end{loglogaxis}
	% \draw [dashed] (2,0.4)--(5,1.6) node[below=1.5pt]{slope 4};
\end{tikzpicture}
\hspace{0.2cm}
\begin{tikzpicture}[scale=0.8]
	\begin{loglogaxis}[
		width=.4\textwidth,
		height=.4\textheight,
		grid=both,
		major grid style={line width=1pt,draw=gray!50},
		grid style = dotted,
		mark size = 2pt,
		%ymin= 1e-15,ymax = 1e-5,
		%xmin = 0, xmax = 8,
		title={$\psi_6$},
		xlabel={$N$}, %ylabel={Relative $L_2$ error},
		ymin = 10^(-9),
		ymax = 10^1,
		ytick = {10^(-8),10^(-6),10^(-4),10^(-2),10^0},
		yticklabels = {$10^{-8}$,$10^{-6}$,$10^{-4}$,$10^{-2}$,$10^0$},
		legend cell align = {left},
		legend pos = south west,
		legend entries={$d=2$,$d=3$,$d=4$,$d=5$}]
		\addplot [thick,dashed, mark = square*,mark options={solid},red] table [x=N,y=Linf_d2] {Data_PeriodKer/sgqi_conv_diffdim_Linf_m6_F2.dat};
		\addplot [thick, dotted, mark=*,mark options={solid}, green] table [x=N,y=Linf_d3] {Data_PeriodKer/sgqi_conv_diffdim_Linf_m6_F2.dat};
		\addplot [thick,mark = triangle*, blue] table [x=N, y=Linf_d4] {Data_PeriodKer/sgqi_conv_diffdim_Linf_m6_F2.dat};
		\addplot [thick,densely dashed, mark = diamond*,mark options={solid}, magenta] table [x=N,y=Linf_d5] {Data_PeriodKer/sgqi_conv_diffdim_Linf_m6_F2.dat};
	\end{loglogaxis}
	% \draw [dashed] (2,0.4)--(5,1.6) node[below=1.5pt]{slope 4};
\end{tikzpicture}
\vspace{-1pt}
\caption{Approximation errors for  $F_6(\T^d)$ on sparse grids for  dimensions $d=2,3,4,5$.}
\label{fig:sgqi_diff_dim_Linf}
\end{figure}

\begin{figure}[htbp]
\centering
\vspace{1cm}
\begin{tikzpicture}[scale=0.8]
	\begin{loglogaxis}[
		grid=both,
		major grid style={line width=1pt,draw=gray!50},
		grid style = dotted,
		mark size = 2.5pt,
		%log ticks with fixed point,
		%ytick distance = 0.001,
		% ymin = 10^(-6),
		% ymax = 10^(0),
		% ytick = {10^(-6),10^(-4),10^(-2),10^(0)},
		% yticklabels = {$10^{-6}$,$10^{-4}$,$10^{-2}$,$10^{0}$},
		xlabel={$N$},
		title={$\psi_{2m+2}$, $d=3$},
		legend cell align = {left},
		legend pos = south west,
		legend entries={$m=0$,$m=1$,$m=2$}]
		\addplot [ultra thick,dashed, mark = square*,mark options={solid},red] table [x=N,y=Linf_m2] {Data_PeriodKer/sgqi_conv_diffOrder_Linf_d3_F2.dat};
		\addplot [ultra thick, dotted, mark=*,mark options={solid}, green] table [x=N,y=Linf_m4] {Data_PeriodKer/sgqi_conv_diffOrder_Linf_d3_F2.dat};
		\addplot [ultra thick,mark = triangle*, blue] table [x=N, y=Linf_m6] {Data_PeriodKer/sgqi_conv_diffOrder_Linf_d3_F2.dat};
	\end{loglogaxis}
	\draw [dashed] (3,3)--(5.5,0.5);
	\node[rectangle] at (2.5,2.5) {slope 1};
\end{tikzpicture}
\hspace{0.5cm}
\begin{tikzpicture}[scale=0.8]
	\begin{loglogaxis}[
		grid=both,
		major grid style={line width=1pt,draw=gray!50},
		grid style = dotted,
		mark size = 2.5pt,
		%log ticks with fixed point,
		%ytick distance = 0.001,
		% ymin = 10^(-6),
		% ymax = 10^(0),
		% ytick = {10^(-6),10^(-4),10^(-2),10^(0)},
		% yticklabels = {$10^{-6}$,$10^{-4}$,$10^{-2}$,$10^{0}$},
		xlabel={$N$},
		title={$\psi_{2m+2}$, $d=5$},
		ylabel={},
		legend cell align = {left},
		legend pos = south west,
		legend entries={$m=0$,$m=1$,$m=2$}]
		\addplot [ultra thick,dashed, mark = square*,mark options={solid},red] table [x=N,y=Linf_m2] {Data_PeriodKer/sgqi_conv_diffOrder_Linf_d5_F2.dat};
		\addplot [ultra thick, dotted, mark=*,mark options={solid}, green] table [x=N,y=Linf_m4] {Data_PeriodKer/sgqi_conv_diffOrder_Linf_d5_F2.dat};
		\addplot [ultra thick,mark = triangle*, blue] table [x=N, y=Linf_m6] {Data_PeriodKer/sgqi_conv_diffOrder_Linf_d5_F2.dat};
	\end{loglogaxis}
	\draw [dashed] (3,3)--(5.5,0.5);
	\node[rectangle] at (2.5,2.5) {slope 1};
\end{tikzpicture}
\vspace{-1pt}
\caption{Approximation errors for  $F_6(\T^d)$ on sparse grids for  dimensions $d=3$ and $d=5$.}
\label{fig.ErrHighDim_Sgqi_Dim_d=5}
\end{figure}

Next, we use  quasi-interpolation on sparse grids  with restricted  generalized Gaussian kernels $\psi_{2m+2}$ ($m=0,1,2$) to approximate the test function $F_6(\T^d)$ for $d=2,3,4,5$, respectively. In addition, we also compare prediction errors at a randomly chosen point $x=(1.44,0.56,4.17)$, running time, and  numbers of grid points between sparse grids and its full grids counterpart in Table $3$ under  $m=0$ and $d=3$ as an example. We can find that  sparse grids quasi-interpolation takes  much less computation time  than its full grids counterpart at high-level grids.

\begin{table}[htbp]
\centering
%\captionsetup{labelsep=period,labelfont=bf}
\caption{Approximation errors and running time for $F_6(\T^3)$ with $m=0$.}
\label{tab:3D_Computation_time}
\begin{tabular}{ccccccccc} % Ensure there are exactly 7 columns defined
\toprule
 & \multicolumn{2}{c}{error} & \multicolumn{2}{c}{running time (/$s$)} & \multicolumn{2}{c}{numbers of grid points}  \\
\cmidrule(lr){2-3} \cmidrule(lr){4-5} \cmidrule(lr){6-7} % Align the small horizontal lines correctly
$l$& sparse  & full & sparse & full & sparse & full  \\
\specialrule{1.5pt}{0pt}{0pt}
$5$& 2.004e-04&7.602e-05 &0.081 &0.015 & 3713 & 32768  \\
$6$  & 9.019e-05 &3.654e-05 &  0.101  & 0.029& 8961 &  262144 \\
$7$ &3.534e-05  & 1.148e-05 & 0.164   &0.196 & 21249 & 2097152\\
$8$ & 1.483e-05 & 3.336e-06 & 0.409&1.653& 49665 & 16777216\\
$9$ & 6.338e-06 & 2.567e-07 &  0.765 & 16.515& 114689&  134217728\\
\bottomrule
\end{tabular}
\end{table}

Relative $L_{\infty}$-norm approximation errors  are provided in  Figure 
 \ref{fig:sgqi_diff_dim_Linf}. We observe a small decrease in convergence rate as the dimension increases. Nevertheless, sparse grids quasi-interpolation with  restricted generalized Gaussian kernels satisfying high-order periodic Strang-Fix conditions do provide better approximation  with  higher convergence rate. More importantly, as demonstrated in Figure \ref{fig.ErrHighDim_Sgqi_Dim_d=5},  there is even an upturn of relative $L_{\infty}$-norm approximation errors in the case of  restricted generalized Gaussian kernel satisfying periodic Strang-Fix condition of order six. In particular,   convergence order in terms of $N$ exceeds one that surpasses the performance of  well-known and commonly used Monte-Carlo  method for high dimensional function approximation.

\textbf{Example 2.}
We consider non-periodic functions over a cube  \cite[Formula (55)]{nasdala2021efficient}
\begin{equation*}
    g(\y)=\prod_{r=1}^d \left(y_r^2-y_r+\frac{3}{4}\right),\ \ \y\in [-\pi,\pi]^d.
\end{equation*} 

We first approximate $g(\y)$  using classical quasi-interpolation scheme defined in \cite[Formula (1-2)]{2017Multilevel} based on discrete function values evaluated at full uniform grids for $d=2$.  As an example, we choose $100$ prediction points including $64$ interior points  and $36$ boundary points  (see  the left Subfigure  of  Figure \ref{fig.ErrHighDim_Sgqi_Dim_d2}). Relative $L_{\infty}$-norm approximation errors  are provided in the right  Subfigure of Figure \ref{fig.ErrHighDim_Sgqi_Dim_d2}. We can observe that the scheme provides an excellent convergence  rate at $64$ interior points. However,  it losses convergence at $100$ prediction points, suggesting that boundary effects significantly affect the approximation accuracy.

\begin{figure}[htbp]
\centering

\begin{minipage}{0.42\linewidth}
\centering
\begin{tikzpicture}[scale=2.3]
\draw[thick] (-1,-1) rectangle (1,1);

\foreach \i in {0,...,9}{
  \foreach \j in {0,...,9}{
    \pgfmathsetmacro\x{-1 + 2*\i/9}
    \pgfmathsetmacro\y{-1 + 2*\j/9}

    \ifnum\i=0\relax
      \filldraw[fill=white] (\x,\y) circle (1pt);
    \else\ifnum\i=9\relax
      \filldraw[fill=white] (\x,\y) circle (1pt);
    \else\ifnum\j=0\relax
      \filldraw[fill=white] (\x,\y) circle (1pt);
    \else\ifnum\j=9\relax
      \filldraw[fill=white] (\x,\y) circle (1pt);
    \else
      \fill[blue!70] (\x,\y) circle (1pt);
    \fi\fi\fi\fi
  }
}

\node[below] at (-1,-1) {$-\pi$};
\node[below] at ( 1,-1) {$\pi$};

\end{tikzpicture}
\end{minipage}
\hspace{0.2cm}
\begin{minipage}{0.55\linewidth}
\centering
\begin{tikzpicture}[scale=0.8]
\begin{loglogaxis}[
	grid=both,
	major grid style={line width=1pt,draw=gray!50},
	grid style=dotted,
	mark size=2.5pt,
	ymin=10^(-7),
	ymax=10^(0),
	xlabel={$N$},
	legend pos=south west,
	legend entries={Boundary and interior points,Interior points}
]
\addplot [ultra thick,dashed,mark=square*,red]
table [x=N,y=Linf_m2] {Data_PeriodKer/sgqi_conv_diffOrder_Linf_d1_non_periodic_relative_error_non_convergence.dat};

\addplot [ultra thick,dotted,mark=*,green]
table [x=N,y=Linf_m4] {Data_PeriodKer/sgqi_conv_diffOrder_Linf_d1_non_periodic_relative_error_non_convergence.dat};

\end{loglogaxis}
\end{tikzpicture}
\end{minipage}

\caption{Prediction points and approximation errors.}
\label{fig.ErrHighDim_Sgqi_Dim_d2}
\end{figure}

To address this problem, we apply quasi-interpolation $Q_{\h,\s}g$ defined in Formula \eqref{np} to approximate $g$. Here we choose  weight function $\omega_r(\gamma_r(x_r))=1$ and the  logarithmic  (torus-to-cube)  transformation function as 
$$
\gamma_r(x_r)=\pi \cdot \frac{(\pi + x_r)^{\eta_r} - (\pi - x_r)^{\eta_r}}{(\pi + x_r)^{\eta_r} + (\pi - x_r)^{\eta_r}},
$$
$$
\gamma_r^{'}(x_r)=\frac{4\pi \eta_r \cdot (\pi + x_r)^{\eta_r - 1} \cdot (\pi - x_r)^{\eta_r - 1}}{\left((\pi + x_r)^{\eta_r} + (\pi - x_r)^{\eta_r}\right)^2}
$$
with $x_r \in [-\pi,\pi], \ \eta_r\in \mathbb{R}_{+}$. 
We first approximate the above non-periodic functions $g$ for $d=1,2$ with $\eta_r=2, \ r=1,\cdots,d$ on regular tensor grids.  The restricted generalized Gaussian kernels $\psi_{2m+2}$ under $m=0,1,2$ are used to compute weighted $L_{\infty}$-norm approximation errors and corresponding convergence rates. Numerical results are  presented in Table \ref{tab:1D_FuncAppr_gauss_non_periodic} and Table \ref{tab:2D_FuncAppr_gauss_non_periodic}, respectively. The results confirm that the convergence rates are consistent with the orders of generalized  Strang–Fix conditions.

\begin{table}[htbp]
\centering
%\captionsetup{labelsep=period,labelfont=bf}
\caption{Weighted $L_{\infty}$-norm approximation errors and convergence rates for $d=1$.}
\label{tab:1D_FuncAppr_gauss_non_periodic}
\begin{tabular}{ccccccc}
\toprule
 & \multicolumn{2}{c}{$m=0$} & \multicolumn{2}{c}{$m=1$} & \multicolumn{2}{c}{$m=2$} \\
\cmidrule(lr){2-3} \cmidrule(lr){4-5} \cmidrule(lr){6-7}
$N$ & error & rate & error & rate & error & rate \\
\specialrule{1.5pt}{0pt}{0pt}
$64$   & 1.065e-02 &      & 1.325e-02 &      & 1.151e-01 &      \\
$128$  & 6.428e-03 & 0.74 & 5.864e-03 & 1.19 & 2.271e-03 & 5.73 \\
$256$  & 5.867e-04 & 2.11 & 5.869e-04 & 2.27 & 1.457e-05 & 6.53 \\
$512$  & 1.421e-04 & 2.23 & 2.763e-06 & 4.03 & 1.679e-07 & 6.59 \\
$1024$ & 2.357e-05 & 2.33 & 6.999e-08 & 4.64 & 1.410e-09 & 6.66 \\
\bottomrule
\end{tabular}
\end{table}

\begin{table}[htbp]
\centering
%\captionsetup{labelsep=period,labelfont=bf}
\caption{Weighted $L_{\infty}$-norm approximation errors and convergence rates for $d=2$.}
\label{tab:2D_FuncAppr_gauss_non_periodic}
\begin{tabular}{ccccccc}
\toprule
 & \multicolumn{2}{c}{$m=0$} & \multicolumn{2}{c}{$m=1$} & \multicolumn{2}{c}{$m=2$} \\
\cmidrule(lr){2-3} \cmidrule(lr){4-5} \cmidrule(lr){6-7}
$l$ & error & rate & error & rate & error & rate \\
\specialrule{1.5pt}{0pt}{0pt}
$7$  & 1.735e-03 &      & 2.127e-03 &      & 1.011e-03 &      \\
$8$  & 3.342e-04 & 2.38 & 1.467e-04 & 3.89 & 2.382e-05 & 5.42 \\
$9$  & 8.061e-05 & 2.21 & 4.868e-06 & 4.40 & 1.173e-07 & 6.54 \\
$10$ & 1.999e-05 & 2.14 & 2.721e-07 & 4.38 & 1.504e-09 & 6.58 \\
$11$ & 4.988e-06 & 2.09 & 1.658e-08 & 4.31 & 2.248e-11 & 6.48 \\
\bottomrule
\end{tabular}
\end{table}

\begin{figure}[htbp]
\centering
\vspace{1cm}
\begin{tikzpicture}[scale=0.8]
	\begin{loglogaxis}[
		grid=both,
		major grid style={line width=1pt,draw=gray!50},
		grid style = dotted,
		mark size = 3pt,
		xmin = 2000, xmax = 200000, 
		ymin = 10^(-11),
		ymax = 2*10^(-2),
		xlabel={$N$},
        title={$\psi_{2m+2}$, $d=2$},
		legend cell align = {left},
		legend pos = south west,
		legend entries={$m=0$,$m=1$,$m=2$}]
		
    \addplot[ultra thick, dashed, mark=square*, mark options={solid}, red] coordinates {
      (2817,9.889e-03)(6145,3.868e-03)(13313,4.710e-04)(28673,6.986e-05)(61441,4.022e-05)(131073,5.721e-06)};
    \addplot[ultra thick, dotted, mark=*, color=green] coordinates {
      (2817,3.643e-03)(6145,1.118e-03)(13313,2.553e-04)(28673,3.410e-06)(61441,2.492e-08)(131073,3.968e-09)};
    \addplot[ultra thick, mark=triangle*, blue] coordinates {
      (2817,2.678e-03)(6145,3.680e-04)(13313,1.181e-04)(28673,6.213e-07)(61441,7.883e-9)(131073,1.384e-10)};
     \addplot[dashed, color=black, thin, forget plot] coordinates {(10000, 4e-3) (50000, 1.6e-4)};
    
    \addplot[dashed, color=black, thin, forget plot] coordinates {(15000, 2.5e-4) (45000, 3.08e-6)};
    \addplot[dashed, color=black, thin, forget plot] coordinates {(20000, 2e-7) (80000, 3.125e-10)};
\end{loglogaxis}
	\node[rectangle] at (5.2,0.7) {\scriptsize slope 6};
	\node[rectangle] at (5.2,3.5) {\scriptsize slope 4};
	\node[rectangle] at (5.2,4.5) {\scriptsize slope 2};
\end{tikzpicture}
\hspace{0.5cm}
\begin{tikzpicture}[scale=0.8]
	\begin{loglogaxis}[
		grid=both,
		major grid style={line width=1pt,draw=gray!50},
		grid style = dotted,
		mark size = 2.5pt,
		%log ticks with fixed point,
		%ytick distance = 0.001,
		% ymin = 10^(-6),
		% ymax = 10^(0),
		% ytick = {10^(-6),10^(-4),10^(-2),10^(0)},
		% yticklabels = {$10^{-6}$,$10^{-4}$,$10^{-2}$,$10^{0}$},
		xlabel={$N$},
		title={$\psi_{2m+2}$, $d=3$},
		ylabel={},
		legend cell align = {left},
		legend pos = south west,
		legend entries={$m=0$,$m=1$,$m=2$}]
		\addplot [ultra thick,dashed, mark = square*,mark options={solid},red] table [x=N,y=Linf_m2] {Data_PeriodKer/sgqi_conv_diffOrder_Linf_d3_non_periodic_relative_error_1.dat};
		\addplot [ultra thick, dotted, mark=*,mark options={solid}, green] table [x=N,y=Linf_m4] {Data_PeriodKer/sgqi_conv_diffOrder_Linf_d3_non_periodic_relative_error_1.dat};
		\addplot [ultra thick,mark = triangle*, blue] table [x=N, y=Linf_m6] {Data_PeriodKer/sgqi_conv_diffOrder_Linf_d3_non_periodic_relative_error_1.dat};

\addplot[dashed, color=black, thin, forget plot]
  coordinates {(1.5e5, 1.555e-2) (9e5, 2.482e-3)};

\addplot[dashed, color=black, thin, forget plot]
  coordinates {(1.5e5, 2.391e-4) (9e5, 1.388e-5)};

\addplot[dashed, color=black, thin, forget plot]
  coordinates {(1.5e5, 1.365e-5) (9e5, 3.318e-7)};

	\end{loglogaxis}
	\node[rectangle] at (4.3,0.6) {\scriptsize slope 2.10};
	\node[rectangle] at (5.2,2.8) {\scriptsize slope 1.78};
	\node[rectangle] at (5.2,4.6) {\scriptsize slope 1.10};
\end{tikzpicture}
\vspace{-1pt}
\caption{Approximation errors for  $g(\y)$ on sparse grids  in  dimensions $d=2$ and $d=3$.}
\label{nonperiodic}
\end{figure}

   Moreover, we also use quasi-interpolation on sparse grids with the restricted generalized Gaussian kernels $\psi_{2m+2}$ under $m=0,1,2$ to approximate the non-periodic functions for $d=2,3$, respectively. Weighted $L_{\infty}$-norm  approximation errors are provided in the left subfigure of  Figure \ref{nonperiodic} for $d=2$ with $\eta_r=2,\ r=1,2$.  Though the convergence rates are slightly less than the orders $2,4$ and $6$,  the scheme still provides high convergence rates.  Finally, we demonstrate relative weighted $L_{\infty}$-norm approximation errors  in the right subfigure of  Figure \ref{nonperiodic} for $d=3$ with $\eta_r=3,\ r=1,2,3$. The results show that the scheme can give a better convergence order than the Monte-Carlo method.

\section{Conclusions and discussions} 
Quasi-interpolation with radial kernels in the Schoenberg's model  often provides much lower approximation order. To circumvent the pitfall, the paper provides a new framework of constructing a quasi-interpolation scheme from generalized Gaussian kernel via the periodization technique and kernel restriction trick.  The key feature of our technique is  that the resulting quasi-interpolation provides the highest approximation order  that is up to the one of generalized Strang-Fix conditions satisfied by  generalized Gaussian kernels. However, the paper is only a starting point taking Gaussian kernels as an example, future works will focus on  thorough and comprehensive study of constructing high-order quasi-interpolation   from general radial kernels.
\label{sec:Conclusion}

\section*{Acknowledgments}We thank two anonymous referees for their constructive comments and 
suggestions that help us improve our paper.
The first author and the second author  are  supported by  Youth Project (Class A) of Anhui Provincial Natural Science Foundation (No. 2508085J009) and National Natural Science Foundation of China (No. 12271002).  The third author is  supported by National Natural Science Foundation of China (No. 12571407, 12101310), Natural Science Foundation of Jiangsu Province (No. BK20252037) and a Jiangsu Shuangchuang Team program (No. JSSCTD202449).

%\section*{References}
\bibliographystyle{plain}%plain  alpha
\bibliography{periodic_kernel}
\appendix 
\renewcommand{\thelemma}{A.\arabic{lemma}}
%\section{Proof of \ref{thm:Integral}}
\section{Proof of Theorem \ref{thm:Integral}}
\label{Appendix:A}

Before proving the theorem, we establish two preliminary lemmas: one concerning the Fourier transform expression of the generalized Gaussian kernel, and another providing a combinatorial identity, both for use in the proof of Theorem \ref{thm:Integral}. 
\begin{lemma}\label{lem:FourierTransform}
	Let $\phi_{2m+2}(r)$ be a two-dimensional generalized Gaussian kernel defined in \eqref{eq:Gauss-SF2m} and $L_j^{(0)}(x)$ be Laguerre polynomial \eqref{eq:GeneralLaguerrePoly} with $\beta=0$, then  we have
	\begin{equation}\label{form}
		\mathcal{F}_2\phi_{2m+2}(r;c)=\sqrt{2\pi}c\sum_{j=0}^m(-1)^j\binom{m + \tfrac{1}{2}}{m - j}L_j^{(0)}({\frac{c^2r^2}{2}})e^{-\frac{c^2r^2}{2}}, r=\|\bm{\omega}\|.
	\end{equation}
\end{lemma}
\begin{proof}
	The definition of $\phi_{2m+2}$ (that is  \eqref{eq:Gauss-SF2m}) together with  \eqref{eq:RBFFourTran} and  \eqref{eq:GeneralLaguerrePoly} leads to 
	\begin{equation}\label{FPsi}
		\begin{aligned}
			\mathcal{F}_2\phi_{2m+2}(r;c) &=2\pi\int_{0}^{\infty}\phi_{2m+2}(t)\cdot t J_0(rt)\d t\\
			&=\int_{0}^{\infty}\frac{\sqrt{2\pi}}{c}L_{m}^{(1/2)}(\frac{t^2}{2c^2})\ e^{-\frac{t^2}{2c^2}}\cdot tJ_0(rt)\d t\\
			&=\frac{\sqrt{2\pi}}{c}\sum_{j=0}^m\frac{(-1)^j}{j!}\binom{m + \frac{1}{2}}{m - j}\int_{0}^{\infty}\frac{t^{2j+1}}{(2c^2)^j}J_0(rt) e^{-\frac{t^2}{2c^2}}\d t.
		\end{aligned}
	\end{equation}
	From \cite[13.3]{watson1966treatise}, for any even integer  $\mu-\nu$ and positive number $p$,   we have the general Hankel's formula 
    \begin{align*}
        \int_{0}^{\infty}J_{\nu}(at)e^{-p^2t^2}\cdot t^{\mu-1}dt
       =~&\frac{\Gamma(\frac{1}{2}\nu+\frac{1}{2}\mu)\cdot (\frac{1}{2}a/p)^{\nu}}{2p^{\mu}\Gamma(\nu+1)}\\
        &~~\cdot e^{-\frac{a^2}{4p^2}}{_{1}F_1}(\frac{1}{2}\nu-\frac{1}{2}\mu+1;\nu+1;\frac{a^2}{4p^2}),
    \end{align*}
	where ${_{1}F_1}$ is the confluent hypergeometric function.
	In particular, in the above formula, we take $$a=r, ~\nu=0,~\mu=2j+2,~p^2=\frac{1}{2c^2}$$ and get
	\begin{equation}\label{special}
		\int_{0}^{\infty}t^{2j+1}J_0(rt) e^{-\frac{t^2}{2c^2}}\d t=\frac{j!(2c^2)^{j+1}}{2}e^{-\frac{c^2r^2}{2}}{_{1}F_1}(-j;1;\frac{c^2r^2}{2}).
	\end{equation}
	Observe that, for a positive integer $j$, the function ${_{1}F_1}(-j;1;z)$ reduces to a Laguerre polynomial,
	\begin{equation}\label{F1}
		{_{1}F_1}(-j;1;z)=L_j^{(0)}(z)=\sum_{k=0}^j\frac{(-1)^k}{k!}\binom{j}{k}z^k.
	\end{equation}
	Consequently, plugging \eqref{special}-\eqref{F1} into \eqref{FPsi}, we have   Identity \eqref{form}.
\end{proof}

\begin{lemma}{(A combinatorial identity)}\label{lem:CombIden}
	Let $k, m$ be two nonnegative integers and $z\in\R$, then the following identity holds,
	$$\sum_{j=k}^m(-1)^j\binom{z}{j-k}=\sum_{j=k}^{m}(-1)^j\binom{m+1-z}{m-j}\binom{j}{k}.$$
\end{lemma}

\begin{proof}
	Since
	\begin{align*}
		\sum_{j=0}^n(-1)^j\binom{z}{j}&=1+\sum_{j=1}^n(-1)^j\left[\binom{z-1}{j}+\binom{z-1}{j-1}\right]\\
		&=\sum_{j=0}^n(-1)^j\binom{z-1}{j}-\sum_{j=0}^{n-1}(-1)^j\binom{z-1}{j}=(-1)^n\binom{z-1}{n},
	\end{align*}
	we have
	$$\sum_{j=k}^m(-1)^j\binom{z}{j-k}=(-1)^k\sum_{j=0}^{m-k}(-1)^j\binom{z}{j}=(-1)^m\binom{z-1}{m-k}.$$
	On the other hand,  with the help of  hypergeometric function, we have
	\begin{align*}
		\sum_{j=k}^m(-1)^j\binom{m+1-z}{m-j}\binom{j}{k}
		=&\sum_{j=0}^{m-k}(-1)^{j+k}\binom{m+1-z}{m-k-j}\binom{j+k}{k}\\
		=&(-1)^k\binom{1+m-z}{m-k}{_{2}F_1}(k-m,1+k,2+k-z;1).
	\end{align*}
	
	Further, according to the definition of ${_{2}F_1}$ and the Chu-Vandermonde convolution formula \cite{askey1975orthogonal}, we obtain
	$${_{2}F_1}(-n,a,b;1)=\sum_{k=0}^n\frac{(-n)_k(a)_k}{k!(b)_k}=\frac{(b-a)_n}{(b)_n},$$
	and thus
	$${_{2}F_1}(k-m,1+k,2+k-z;1)=\frac{(1-z)_{m-k}}{(2+k-z)_{m-k}}.$$
	Moreover, noting that
	\begin{align*}
		\sum_{j=k}^m(-1)^j\binom{m+1-z}{m-j}\binom{j}{k}
		=&(-1)^k\binom{1+m-z}{m-k}\cdot \frac{(1-z)_{m-k}}{(2+k-z)_{m-k}}\\
		=&(-1)^k\frac{(1-z)_{m-k}}{(m-k)!}=(-1)^m\binom{z-1}{m-k},
	\end{align*}
	we complete the proof.
\end{proof}

We then recall two additional lemmas from the literature, which we state here for completeness and later use.
\begin{lemma}
	{(\cite[Formula 9.7.1]{Abramowitz1964handbook})}\label{lem:HankelExpan}
	Let $ I_{\ell}(z)$ be the modified Bessel function of the first kind. Then, for each fixed $\ell$ and large $z$,  $ I_{\ell}(z)$ has the Hankel’s asymptotic expansion
	\begin{equation*}\label{coeff}
		I_{\ell}(z)\sim \frac{e^{z}}{\sqrt{2\pi z}}\sum_{\gamma=0}^{\infty}(-1)^{\gamma}a_{\gamma}(\ell)z^{-\gamma},~~a_{\gamma}(\ell)=\frac{(\frac{1}{2}-\ell)_{\gamma}(\frac{1}{2}+\ell)_{\gamma}}{(-2)^{\gamma} \gamma!},
	\end{equation*}
	in which $(x)_{\gamma}=\prod_{j=0}^{{\gamma}-1}(x-j)$ represents the rising factorial of $x$. 
	
	Specifically, by taking $\mu=4\ell^2$, we have
	\begin{equation*}\label{eq:ModifiedBesselExpansion}
		I_{\ell}(z)\sim \frac{e^z}{\sqrt{2\pi z}}\left[1-\frac{\mu-1}{8z}+\frac{(\mu-1)(\mu-9)}{2!(8z)^2}-\frac{(\mu-1)(\mu-9)(\mu-25)}{3!(8z)^3}+\ldots\right]. 
	\end{equation*}
\end{lemma}

\begin{lemma}{(\cite[Lemma 3.1]{gao2017constructing})}\label{lem:Polyn}
	Let $g(t)=\sum_{l=0}^{m+1} a_l t^l$ be a univariate polynomial of degree $m+1$, then the polynomial defined by 
	$$G_m(t)=\sum_{j=0}^m(-1)^j\frac{t^j}{j!}g^{(j)}(t)$$
	satisfies the condition
	$$G_m(t)=a_0+\O(t^{m+1}),~~t\rightarrow 0.$$
\end{lemma}

With these results in place, we now present the proof.
\begin{proof}[Proof of Theorem \ref{thm:Integral}]
We begin by showing the following identity 
\begin{equation}\label{eq:identity_high_low}
    \phat_{2m+2}(\ell;\rho)=\sum_{j=0}^m\frac{(-\rho)^j}{j!}\frac{d^j}{d\rho^j}\big(\phat_{2}(\ell;\rho)\big), \ \rho=c^2.
\end{equation}

Using Lemma \ref{lem:FourierTransform} and Lemma \ref{lem:FourierCoeff} for $d=2$,  we have
	\begin{equation*}
		\begin{aligned}
			\phat_{2m+2}(\ell;\rho)&=\int_{0}^{\infty}tJ_{\ell}^2(t) \mathcal{F}_2\phi_{2m+2}(t)\d t.
		\end{aligned}
	\end{equation*}
Based on the explicit form of $\mathcal{F}_2\phi_{2m+2}$ in terms of Laguerre polynomial, we expand $\phat_{2m+2}(\ell;\rho)$ as
\begin{equation*}
\begin{aligned}
\phat_{2m+2}(\ell;\rho)= &\sqrt{2\pi}\rho^{1/2}\sum_{j=0}^m(-1)^j\binom{m+\tfrac{1}{2}}{m-j}\int_{0}^{\infty}tJ_{\ell}^2(t)L_j^{(0)}({\frac{\rho t^2}{2}})e^{-\frac{\rho t^2}{2}}\d t\\
    =&\sqrt{2\pi}\rho^{1/2}\sum_{k=0}^m\frac{(-1)^k}{k!}\sum_{j=k}^{m}(-1)^j\binom{m+\tfrac{1}{2}}{m-j}\binom{j}{k}\int_{0}^{\infty}tJ_{\ell}^2(t)({\frac{\rho t^2}{2}})^k e^{-\frac{\rho t^2}{2}}\d t,
\end{aligned}
\end{equation*}
after reindexing the summation over $j$ and $k$.
    
Moreover,  by Lemma \ref{lem:CombIden} with $z = 1/2$, we have
    \begin{equation*}
		\sum_{j=k}^m(-1)^j\binom{\tfrac{1}{2}}{j - k}
=\sum_{j=k}^{m}(-1)^j\binom{m + \tfrac{1}{2}}{m - j} \binom{j}{k}
.
	\end{equation*} 
 This in turn leads to
    \begin{equation*}
    \begin{aligned}
        \phat_{2m+2}(\ell;\rho)=&\sqrt{2\pi}\sum_{k=0}^m\frac{(-1)^k}{k!}\int_{0}^{\infty}tJ_{\ell}^2(t)\sum_{j=k}^{m}(-1)^j\binom{\tfrac{1}{2}}{j - k}
({\frac{\rho t^2}{2}})^k \rho^{1/2}e^{-\frac{\rho t^2}{2}}\d t\\
                =&\sqrt{2\pi}\sum_{j=0}^m\frac{(-1)^j}{j!}\int_{0}^{\infty}tJ_{\ell}^2(t)\sum_{k=0}^j\binom{j}{k}\big(\frac{1}{2}\big)_{j-k}\rho^k(-\frac{t^2}{2}\big)^k \rho^{1/2}e^{-\rho\frac{t^2}{2}}\d t.
    \end{aligned}	
	\end{equation*}
    Furthermore,  according to the Leibniz rule for differentiation, that is, 
    $$\frac{d^j}{d\rho^j}\left(\rho^{1/2}e^{-\rho\frac{t^2}{2}}\right)=\sum_{k=0}^j\binom{j}{k}\big(\frac{1}{2}\big)_{j-k}\rho^{1/2-(j-k)}\big(-\frac{t^2}{2}\big)^ke^{-\rho\frac{t^2}{2}},$$
    we obtain
    \begin{equation*}
		\begin{aligned} 
			\phat_{2m+2}(\ell;\rho)&=\sqrt{2\pi}\int_{0}^{\infty}tJ_{\ell}^2(t)\sum_{j=0}^m\frac{(-\rho)^j}{j!}\frac{d^j}{d\rho^j}\left(\rho^{1/2}e^{-\rho\frac{t^2}{2}}\right)\d t.\\
            	&=\sum_{j=0}^m\frac{(-\rho)^j}{j!}\frac{d^j}{d\rho^j}\Big(\sqrt{2\pi}\rho^{1/2}\int_{0}^{\infty}tJ_{\ell}^2(t)e^{-\rho\frac{t^2}{2}}\d t\Big)\\
                &=\sum_{j=0}^m\frac{(-\rho)^j}{j!}\frac{d^j}{d\rho^j}\big(\phat_{2}(\ell;\rho)\big).\\
		\end{aligned}
	\end{equation*}
   
Then we show property (i) holds.

The Fourier coefficient $\phat_{2}(\ell;\rho)$ has the exact form \cite{hubbert2002radial,narcowich2007approximation, watson1966treatise}:
	$$\phat_{2}(\ell;\rho)=\sqrt{2\pi}\rho^{-1/2} e^{-\frac{1}{\rho}} I_{\ell}\left(\frac{1}{\rho}\right),$$
    where $I_{\ell}$ is the modified Bessel function of the first kind.  This together with  Lemma \ref{lem:HankelExpan} gives the asymptotic expansion
    $$\phat_{2}(\ell;\rho)\sim\sum_{\gamma=0}^{\infty}(-1)^{\gamma}a_{\gamma}(\ell)\rho^{\gamma},\quad \text{as}\ \rho\rightarrow 0.$$
Thus, for fixed $\ell$, an application of \cite[(95)]{nemes2017error} yields
	$$\phat_{2}(\ell;\rho)= \sum_{\gamma=0}^{N}(-1)^{\gamma}a_{\gamma}(\ell)\rho^{\gamma}+\mathcal{R}_{N+1}(\rho,\ell),\quad \text{with}~~|\mathcal{R}_{N+1}(\rho,\ell)|\leq C_{N+1}|a_{N+1}(\ell)\rho^{N+1}|,$$
    for any  positive integer  $N$ and  any sufficiently small  $\rho=c^2$.
 Therefore, based on  Identity \eqref{eq:identity_high_low} and  Lemma \ref{lem:Polyn},  we have 
$$\phat_{2m+2}(\ell; \rho)=a_0(\ell)+\mathcal{O}(a_{m+1}(\ell)\rho^{m+1}), ~~\rho\rightarrow 0$$ holds true for any $N\geq m+1$.
In addition,  according to \eqref{coeff},  we have  $a_0(\ell)=1$ and  $$a_{m+1}(\ell)\rho^{m+1}=\mathcal{O}(\ell^{2m+2}c^{2m+2}),  ~~c\rightarrow 0.$$
Consequently,  property (i) holds.

We end the proof with showing property (ii) holds. 

For fixed $\rho=c^2>0$ and sufficiently large $\ell$, according to \cite{Abramowitz1964handbook,hubbert2002radial}, we have
	$$I_{\ell}\left(\frac{1}{\rho}\right)\sim \frac{1}{\sqrt{2\pi \ell}}\left(\frac{e}{2\rho \ell}\right)^{\ell},$$
	which in turn leads to
	$$\phat_2(\ell;\rho)=\sqrt{2\pi}\rho^{-1/2} e^{-\frac{1}{\rho}} I_{\ell}\left(\frac{1}{\rho}\right)\sim\frac{e^{-\frac{1}{c^2}}}{2^{\ell}c^{2\ell+1}\sqrt{\ell}}\left(\frac{e}{\ell}\right)^{\ell}.$$
 Thus,  $\phat_2(\ell;\rho)$  decay exponentially for sufficiently large $\ell$. 	Moreover,  the identity   $$\frac{\partial^j I_{\ell}(z)}{\partial z^j}=2^{-j}\sum_{k=0}^j\binom{j}{k}I_{2k-j+\ell}(z)$$ together with Identity  \eqref{eq:identity_high_low}
	yields
	$$\phat_{2m+2}(\ell;\rho)=e^{-\frac{1}{\rho}}\cdot \sum_{j=0}^m \sum_{k=0}^j p_{jk}(\rho^{-1/2})I_{2k-j+\ell}(\frac{1}{\rho}),$$
	where $p_{jk}$ are some polynomials. Thus,  we can verify  that  coefficient $\phat_{2m+2}(\ell;\rho)$ also decays exponentially for sufficiently large  $\ell$.
\end{proof}

\section{Technical results}
\renewcommand{\thelemma}{B.\arabic{lemma}}
 We first provide some notations.  Define \begin{align*}
  &\gamma_2:=\max_{\k_0\in \JN} \left(\frac{1}{N^2}+\left|\frac{\k_0}{N}\right|_2^2\right)^{(s+\tau) /2}< \infty,\\
&\gamma_3: = \left\{
\begin{array}{ll}
\max_{\k_0\in \JN}  \left(\sum_{\bl \in \mathbb{Z}^d} \left|\frac{\k_0}{N}+\bl \right|_2^{-\mu p}\right)^{1/p}< \infty, & \text{if } q > 1,\ 1/q+1/p=1, \mu >d(1-1/q),\\
\max_{\k_0\in \JN}  \sup_{\bl \in \mathbb{Z}^d} \left|\frac{\k_0}{N}+\bl \right|_2^{-\mu }< \infty, & \text{if } q = 1, \mu >0.
\end{array}
\right.
\end{align*}
Then we get the following lemma, which is  applied  in the proof of Theorem \ref{maintheorem}.
\begin{lemma}\label{E_1}
Let $\gamma_1$, $\mathcal{E}_1$ and $\mathcal{E}_2$ be defined in \eqref{eq:gamma_1}, \eqref{eq:E_1} and \eqref{eq:E_2}, respectively.  Let $\bpsi_{\h,\s}$ be an isotropic kernel defined in \eqref{eq:IsoTPKer} and satisfy the periodic Strang-Fix conditions of order $s>d/2$. Then, under the assumptions of Lemma \ref{lem:ConvCardInterp} together with $\sigma=\min\{\mu,s+\tau\}$,  the inequalities  
   \begin{align*}
   \mathcal{E}_1\leq \gamma_2^q 
    \ \gamma_3^q \ \|f\|_{A_q^{\mu}(\mathbb{T}^d)}^q, \ \ 
   \mathcal{E}_2 \leq C \ \gamma_1^q \  \gamma_2^q \  \gamma_3^q \  \|f\|_{A_q^{\mu}(\T^d)}^q,
\end{align*}
hold for $1\leq q < \infty$.
\end{lemma}
\begin{proof}
  We first consider the case $1<q<\infty$. Applying H\"older's inequality to $\mathcal{E}_1$, we have 
\begin{align*}
\mathcal{E}_1
%\leq &~N^{\sigma q-\tau q-sq}\sum_{\k_0\in \JN}|\k_0|_2^{sq}\cdot  (1+|\k_0|_2^2)^{\tau q/2} \cdot \Big(\sum_{\bl \in \mathbb{Z}^d}  (1+|\k_0+\bl N|_2^2) ^{-\mu p/2}\Big)^{q/p} \\
%&\cdot \Big(\sum_{\bl \in \mathbb{Z}^d}   (1+|\k_0+\bl N|_2^2) ^{\mu q/2}|\fhat(\k_0+\bl N)|^q\Big)\\
\leq   & \sum_{\k_0\in \JN} \underbrace{\Big(|\k_0|_2^{sq}\cdot  (1+|\k_0|_2^2)^{\tau q/2} \cdot N^{-\tau q-sq}\Big)}_{\mathcal{H}_1}\cdot \underbrace{\Big(\sum_{\bl \in \mathbb{Z}^d} (1+|\k_0+\bl N|_2^2) ^{-\mu p/2}N^{\sigma p}\Big)^{q/p}}_{\mathcal{H}_2} \\
&\cdot\Big(\sum_{\bl \in \mathbb{Z}^d}  (1+|\k_0+\bl N|_2^2) ^{\mu q/2}|\fhat(\k_0+\bl N)|^q\Big),\ 1/q+1/p=1.
\end{align*}
Further, we have 
\begin{align}
\label{eq:H_1}
 \mathcal{H}_1 \leq   \max_{\k_0 \in \JN}(1+|\k_0|_2^2)^{(s+\tau)q/2}N^{-\tau q-sq}\leq \max_{\k_0\in \JN} \left(\frac{1}{N^2}+\left|\frac{\k_0}{N}\right|_2^2\right)^{(s+\tau)q /2} \leq\gamma_2^q.
\end{align}
Since we assume that $\sigma \leq \mu$, we have 
\begin{equation}
\label{eq:H_2}
\mathcal{H}_2 \leq \max_{\k_0\in \JN}  \Big(\sum_{\bl \in \mathbb{Z}^d} (1+|\k_0+\bl N|_2^2) ^{-\mu p/2}N^{\mu p}\Big)^{q/p}\leq \max_{\k_0\in \JN}  \left(\sum_{\bl \in \mathbb{Z}^d} \left|\frac{\k_0}{N}+\bl \right|_2^{-\mu p}\right)^{q/p}\leq \gamma_3^q.
\end{equation}
These in turn leads to 
\begin{align*}
     \mathcal{E}_1\leq &~ \gamma_2^q \ \gamma_3^q  \sum_{\k_0\in \JN}\sum_{\bl \in \mathbb{Z}^d}  (1+|\k_0+\bl N|_2^2) ^{\mu q/2}|\fhat(\k_0+\bl N)|^q\\
     = &~ \gamma_2^q \ \gamma_3^q \sum_{\k\in \mathbb{Z}^d} (1+|\k |_2^2) ^{\mu q/2}|\fhat(\k )|^q\\
     =&~ \gamma_2^q \ \gamma_3^q\ \|f\|_{A_q^{\mu}(\mathbb{T}^d)}^q,
\end{align*}
where we have used the definition of $\|f\|_{A_q^{\mu}(\mathbb{T}^d)}$ in \eqref{eq:AqalpSpace}.

On the other hand, applying    H\"older's inequality again to  $\mathcal{E}_2$, we can get
\begin{equation*}
    \begin{split}
    	\mathcal{E}_2
   \leq &\sum_{\k_0\in \JN}  \sum_{\bnu\in\Z^d\backslash \{\bm{0}\}}  b_{\bnu}^q|\k_0|_2^{sq}N^{-(s+\tau) q }N^{-\tau q }(1+|\k_0+\bnu N|_2^2)^{\tau q/2}\cdot\\
    &\!\!\!\!\!\!\!\!\!\!\!\!\Big( \sum_{\bl \in \mathbb{Z}^d} (1+|\k_0+\bnu N+\bl N|_2^2)^{-\mu p/2}N^{\sigma p}\Big) ^{q/p}
   \sum_{\bl \in \mathbb{Z}^d} (1+|\k_0+\bnu N+\bl N|_2^2)^{\mu q/2}|\fhat(\k_0+\bnu N+\bl N)|^q.
    \end{split}
\end{equation*}
Based on the following identity
\begin{equation*}
    \begin{split}
  &  \Big( \sum_{\bl \in \mathbb{Z}^d} (1+|\k_0+\bnu N+\bl N|_2^2)^{-\mu p/2}\Big) ^{q/p}\sum_{\bl \in \mathbb{Z}^d} (1+|\k_0+\bnu N+\bm{\ell}N|_2^2)^{\mu q/2}|\fhat(\k_0+\bnu N+\bm{\ell}N)|^q\\%=&~~ \Big( \sum_{\bl \in \mathbb{Z}^d} (1+|\k_0+(\bnu+\bl)N|_2^2)^{-\mu p/2}\Big) ^{q/p}\sum_{\bl \in \mathbb{Z}^d}(1+|\k_0+(\bnu+\bl) N|_2^2)^{\mu q/2}|\fhat(\k_0+(\bnu+\bl) N)|^q\\
  =&~~ \Big( \sum_{\bl' \in \mathbb{Z}^d} (1+|\k_0+\bl' N|_2^2)^{-\mu p/2}\Big) ^{q/p}\sum_{\bl' \in \mathbb{Z}^d}(1+|\k_0+\bl' N|_2^2)^{\mu q/2}|\fhat(\k_0+\bl' N )|^q, \  \bl':=\bl+\bnu, 
    \end{split}
\end{equation*}
we obtain 
\begin{equation*}
    \begin{split}
    	\mathcal{E}_2
   = %&  \sum_{\k_0\in \JN} \Big(  \sum_{\bl' \in \mathbb{Z}^d} (1+|\k_0+\bl' N|_2^2)^{-\mu p/2}N^{\sigma p }\Big) ^{q/p}\sum_{\bl' \in \mathbb{Z}^d} (1+|\k_0+\bl' N|_2^2)^{\mu q/2} |\fhat(\k_0+\bl' N)|^q  \\
 %&\!\!\!\!\!\!\!\!\!\cdot \sum_{\bnu\in\Z^d\backslash \{\bm{0}\}}\mathcal{B}_{\bnu}^q N^{-\tau q }\cdot (1+|\k_0+\bnu N|_2^2)^{\tau q/2}|\k_0|_2^{sq}\cdot  (1+|\k_0|_2^2)^{\tau q/2}\cdot N^{-(s+\tau)q}(1+|\k_0|_2^2)^{-\tau q/2}\\
  &  \sum_{\k_0\in \JN} \mathcal{H}_1\mathcal{H}_2\sum_{\bl' \in \mathbb{Z}^d} (1+|\k_0+\bl' N|_2^2)^{\mu q/2} |\fhat(\k_0+\bl' N)|^q  \\
 &\!\!\!\!\!\!\!\!\!\cdot \underbrace{\sum_{\bnu\in\Z^d\backslash \{\bm{0}\}}b_{\bnu}^q} N^{-\tau q }\cdot (1+|\k_0+\bnu N|_2^2)^{\tau q/2}_{\mathcal{H}_3}\cdot(1+|\k_0|_2^2)^{-\tau q/2},
    \end{split}
\end{equation*}
where we have used $(1+|\k_0|_2^2)^{\tau q/2}(1+|\k_0|_2^2)^{-\tau q/2}=1.$
Note that $$\max_{\k_0\in\JN}(1+|\k_0|_2^2)^{-\tau q/2}=1.$$ Combining \eqref{eq:H_1} with \eqref{eq:H_2}, we obtain
\begin{align*}
   \mathcal{E}_2 \leq &~ \gamma_2^q \ \gamma_3^q \sum_{\k_0\in \JN}\sum_{\bl' \in \mathbb{Z}^d} (1+|\k_0+\bl' N|_2^2)^{\mu q/2} |\fhat(\k_0+\bl' N)|^q \cdot \mathcal{H}_3
\end{align*}
Furthermore, based on Jensen inequality and $\k_0 \in \JN$, we have 
\begin{align*}
    1+|\frac{\k_0}{N}+\bnu |_2^2\leq 1+2(|\bnu|_2^2+|\frac{\k_0}{N}|_2^2)\leq 1+2|\bnu|_2^2+\frac{d}{2}\le C(1+|\bnu|_2^2),
\end{align*}
which in turn leads to 
\begin{align*}
\mathcal{H}_3 \leq&~~\max_{\k_0\in\JN}\Big(\sum_{\bnu\in\Z^d\backslash \{\bm{0}\}}  b_{\bnu}^q~(\frac{1}{N^2}+|\frac{\k_0}{N}+\bnu |_2^2)^{\tau q/2}\Big)\\
\leq &~~\max_{\k_0\in\JN}\Big(\sum_{\bnu\in\Z^d\backslash \{\bm{0}\}}  b_{\bnu}^q~(1+|\frac{\k_0}{N}+\bnu |_2^2)^{\tau q/2}\Big)
	\leq ~ ~C\gamma_1^q.
\end{align*}
 Combining these results with the definition of $\|f\|_{A_q^{\mu}(\mathbb{T}^d)}$, we have 
\begin{align*}
 \mathcal{E}_2\leq &~C \gamma_1^q \  \gamma_2^q \ 
 \gamma_3^q \sum_{\k_0\in \JN}\sum_{\bl' \in \mathbb{Z}^d} (1+|\k_0+\bl' N|_2^2)^{\mu q/2} |\fhat(\k_0+\bl' N)|^q\\
 = &~C \gamma_1^q \  \gamma_2^q \ 
 \gamma_3^q  \sum_{\k\in \mathbb{Z}^d} (1+|\k |_2^2) ^{\mu q/2}|\fhat(\k )|^q\\
 = &~C \gamma_1^q \  \gamma_2^q \ 
 \gamma_3^q \  \|f\|_{A_q^{\mu}(\T^d)}^q.
\end{align*}
 Similarly, we can compute the case $q=1$ to get desired result. 
\end{proof}

\end{document}